\documentclass[leqno,12pt]{amsart} 
\usepackage{amssymb}
\theoremstyle{plain} 
\newtheorem{theorem}{\indent\sc Theorem}[section]
\newtheorem{lemma}[theorem]{\indent\sc Lemma}
\newtheorem{corollary}[theorem]{\indent\sc Corollary}
\newtheorem{proposition}[theorem]{\indent\sc Proposition}

\theoremstyle{definition} 
\newtheorem{definition}[theorem]{\indent\sc Definition}
\newtheorem{remark}[theorem]{\indent\sc Remark}
\newtheorem{example}[theorem]{\indent\sc Example}

\numberwithin{equation}{section}
\def\nmo{{n\hn-\nh1}}

\def\hyp{\hskip.5pt\vbox
{\hbox{\vrule width2.5ptheight0.5ptdepth0pt}\vskip2pt}\hskip.5pt}

\def\hs{\hskip.7pt}
\def\hh{\hskip.4pt}
\def\nh{\hskip-.7pt}

\font\smallit=ptmri at 11.96pt
 1

\def\bg{\varPi}
\def\rd{\delta}

\def\lp{\alpha}
\def\stb{\varSigma}
\def\sg{\mathcal{G}}
\def\tvs{\mathcal{V}}
\def\tws{\mathcal{W}}
\def\w{^{\phantom i}}
\def\nnh{\hskip-1.5pt}
\def\hn{\hskip-.4pt}

\def\bbR{\mathrm{I\!R}}
\def\rn{\bbR\nh^n}
\def\bbQ{{\mathchoice{\setbox0=\hbox{$\displaystyle\rm Q$}\hbox{\raise
0.15\ht0\hbox to0pt{\kern0.4\wd0\vrule height0.8\ht0\hss}\box0}}
{\setbox0=\hbox{$\textstyle\rm Q$}\hbox{\raise
0.15\ht0\hbox to0pt{\kern0.4\wd0\vrule height0.8\ht0\hss}\box0}}
{\setbox0=\hbox{$\scriptstyle\rm Q$}\hbox{\raise
0.15\ht0\hbox to0pt{\kern0.4\wd0\vrule height0.7\ht0\hss}\box0}}
{\setbox0=\hbox{$\scriptscriptstyle\rm Q$}\hbox{\raise
0.15\ht0\hbox to0pt{\kern0.4\wd0\vrule height0.7\ht0\hss}\box0}}}}
\newcommand{\bbC}{{\mathchoice {\setbox0=\hbox{$\displaystyle\mathrm{C}$}
\hbox{\hbox to0pt{\kern0.4\wd0\vrule height0.9\ht0\hss}\box0}} 
{\setbox0=\hbox{$\textstyle\mathrm{C}$}\hbox{\hbox 
to0pt{\kern0.4\wd0\vrule height0.9\ht0\hss}\box0}} 
{\setbox0=\hbox{$\scriptstyle\mathrm{C}$}\hbox{\hbox 
to0pt{\kern0.4\wd0\vrule height0.9\ht0\hss}\box0}} 
{\setbox0=\hbox{$\scriptscriptstyle\mathrm{C}$}\hbox{\hbox 
to0pt{\kern0.4\wd0\vrule height0.9\ht0\hss}\box0}}}}

\def\bbZ{\mathsf{Z\hskip-5ptZ}}
\def\dimr{\dim_{\hskip.3pt\bbR\hskip-1.7pt}\w}
\def\dimq{\dim_{\hskip.8pt\bbQ\hskip-1.7pt}\w}
\def\dimz{\dim_{\hskip-.2pt\mathsf{Z\hskip-3.5ptZ}\hskip-1.7pt}\w}
\def\spr{\mathrm{span}_{\hskip-.2pt\bbR\hskip-1.7pt}\w}

\def\w{^{\phantom i}}

\title[Flat manifolds and reducibility]{Flat manifolds and reducibility}
\author{Andrzej Derdzinski \and Paolo Piccione}
\address{
\begin{tabular}{lll}
Department of Mathematics&&Departamento de Matem\'atica\\
The Ohio State University&&Instituto de Matem\'atica e Estat\'\i stica\\
231 W. 18th Avenue&&Universidade de S\~ao Paulo\\
Columbus, OH 43210&&Rua do Mat\~ao 1010, CEP 05508-900\\
United States of America&&S\~ao Paulo, SP, Brazil\\
{\tt andrzej@math.ohio-state.edu}&&{\tt piccione@ime.usp.br}\\
\end{tabular}
}
\begin{document}
\begin{abstract}
Hiss and Szczepa\'nski proved in 1991 that the holonomy group of any compact 
flat Riemannian manifold, of dimension at least two, acts reducibly on the 
rational span of the Euclidean lattice associated with the manifold via the 
first Bieberbach theorem. Geometrically, their result states that such a 
manifold must admit a nonzero proper parallel distribution with compact 
leaves. We study algebraic and geometric properties of the sublattice-spanned 
holonomy-invariant subspaces that exist due to the above theorem, and of the 
resulting compact-leaf foliations of compact flat manifolds. The class 
consisting of the former subspaces, in addition to being closed under spans 
and intersections, also turns out to admit (usually nonorthogonal) 
complements. As for the latter foliations, we provide descriptions, first -- 
and foremost -- of the intrinsic geometry of their generic leaves in terms of 
that of the original flat manifold and, secondly -- as an essentially obvious 
afterthought -- of the leaf-space orbifold. The general conclusions are then 
illustrated by examples in the form of generalized Klein bottles.
\end{abstract}
\maketitle

\section{Introduction}\label{in}
As shown by Hiss and Szczepa\'nski 
\cite[the corollary in Sect.\ 1]{hiss-szczepanski}, on any compact flat 
Riemannian manifold $\,\mathcal{M}\hs$ with $\,\dim\mathcal{M}=n\ge2\,$ there 
exists a parallel distribution $\,D\hs$ of dimension $\,k$, where $\,0<k<n$, 
such that the leaves of $\,D\hs$ are all compact. Their result, in its 
original algebraic phrasing (see the Appendix), stated that the holonomy group 
$\,H$ of $\,\mathcal{M}\,$ must act reducibly on $\,L\otimes\hn\bbQ$, for
the Euclidean lattice $\,L\,$ corresponding to $\,\mathcal{M}$ (which is a 
maximal A\-bel\-i\-an sub\-group of the fundamental group $\,\bg\hs$ of 
$\,\mathcal{M}$).

The present paper explores the algebraic context and geometric consequences of 
this fact. We view $\,L\,$ as an additive sub\-group of a Euclidean vector 
space $\,\tvs\hs$ (so that $\,L\otimes\hn\bbQ$ becomes identified with the 
rational span of $\,L\,$ in $\,\tvs$), and use the term 
$\,L${\smallit-sub\-space\/} when referring to a vector sub\-space of 
$\,\tvs\hs$ spanned by some subset of $\,L$.

Hiss and Szczepa\'nski's theorem amounts to the existence a nonzero proper 
$\,H\nh$-in\-var\-i\-ant $\,L$-sub\-space $\,\tvs\hh'\subseteq\tvs\nnh$. We 
begin by observing that the class of $\,H\nh$-in\-var\-i\-ant 
$\,L$-sub\-spaces of $\,\tvs\hs$ is closed under the span and intersection 
operations applied to its arbitrary sub\-class\-es (Lemma~\ref{spint}), while 
every $\,H\nh$-in\-var\-i\-ant $\,L$-sub\-space of $\,\tvs\hs$ has an 
$\,H\nh$-in\-var\-i\-ant $\,L$-sub\-space complementary to it 
(Theorem~\ref{invcp}).

The Bie\-ber\-bach group of a given compact flat Riemannian manifold 
$\,\mathcal{M}\hs$ is its fundamental group $\,\bg\hs$ treated as the deck 
transformation group acting via af\-fine isometries on the Euclidean af\-fine 
space $\,\mathcal{E}\hs$ that constitutes the Riemannian universal covering 
space of $\,\mathcal{M}$. The space $\,\tvs\hs$ mentioned above is associated 
with $\,\mathcal{E}\hs$ by being its translation vector space, that is, the 
space of parallel vector fields on $\,\mathcal{E}\nnh$, and the roles of the 
lattice $\,L\,$ and  holonomy group $\,H\hs$ are summarized by the short exact 
sequence $\,L\to\hs\bg\to H$. See Section~\ref{bg}. We proceed to describe, 
in Sections~\ref{lr} --~\ref{gg}, the constituents 
$\,L\nh'\nnh,\bg\hn'\nnh,H\nh'$ appearing in the analog 
$\,L\nh'\nh\to\hs\bg\hn'\nh\to H\nh'$ of this short exact sequence for any 
(compact, flat) leaf $\,\mathcal{M}\hn'$ of a parallel distribution $\,D\nh$, 
guaranteed to exist on $\,\mathcal{M}\hs$ by the aforementioned result of 
\cite{hiss-szczepanski}. Specifically, by Theorem~\ref{restr}(ii), $\,\bg\hn'$ 
(or, $\,L\nh'$) may be identified with a sub\-group of $\,\bg\hs$ (or, 
$\,\tvs$), and $\,H\nh'$ with a homo\-mor\-phic image of a sub\-group of 
$\,H\nnh$. This description becomes particularly simple for leaves 
$\,\mathcal{M}\hn'$ which we call {\smallit generic\/} (Theorem~\ref{gnric}): 
their union is an open dense subset of $\,\mathcal{M}$, they all have the same 
triple $\,L\nh'\nnh,\bg\hn'\nnh,H\nh'\nnh$, and are mutually isometric. When 
all leaves of $\,D\,$ happen to be generic, they form a locally trivial bundle 
with compact flat manifolds serving both as the base and the fibre (the 
{\smallit fibration case}).

Aside from the holonomy group $\,H\nh'$ of each individual leaf 
$\,\mathcal{M}\hn'$ of $\,D\nh$, forming a part of its intrinsic 
(sub\-man\-i\-fold) geometry, $\,\mathcal{M}\hn'$ also gives rise to two 
``extrinsic'' holonomy groups, one arising since $\,\mathcal{M}\hn'$ is a leaf 
of the foliation $\,F\hskip-3.8pt_{\mathcal{M}}\w$ of $\,\mathcal{M}\,$ 
tangent to $\,D\nh$, the other coming from the normal connection of 
$\,\mathcal{M}\hn'\nnh$. Due to flatness of the normal connection, the two 
extrinsic holonomy groups coincide, and are trivial for all generic leaves. In 
Section~\ref{ls} we briefly discuss the leaf space 
$\,\mathcal{M}/\nh F\hskip-3.8pt_{\mathcal{M}}\w$, pointing out that (not 
surprisingly!)
\begin{equation}\label{lsp}
\begin{array}{l}
\mathcal{M}/\nh F\hskip-3.8pt_{\mathcal{M}}\w\mathrm{\,forms\ a\ flat\ 
compact\ or\-bi\-fold,\nnh\ canonically\ identified\ with\ the\ quotient}\\
{}\mathrm{of\ the\ torus\ 
}\,\hs[\mathcal{E}\nnh/\hn\tvs\hh']/[L\cap\tvs\hh']\hs\mathrm{\ under\ 
the\ isometric\ action\ of\ the\ finite\ group\ }\hs\,H\nh.
\end{array}
\end{equation}
In the fibration case (see above), 
$\,\mathcal{M}/\nh F\hskip-3.8pt_{\mathcal{M}}\w$ is the base manifold of the 
bundle.

We illustrate the above conclusions by examples (generalized Klein bottles, 
Section~\ref{gk}), where both the fibration and non-fibration cases occur, 
depending on the choice of $\,D\nh$. 

Section~\ref{ig} provides a formula for the intersection number of generic 
leaves of the foliations of the compact flat manifold $\,\mathcal{M}\,$ 
arising from two mutually complementary $\,H\nh$-in\-var\-i\-ant 
$\,L$-sub\-spaces of $\,\tvs\hs$ (cf.\ Theorem~\ref{invcp}, mentioned earlier).

Both authors' research was supported in part by a FAPESP\nh-\hs OSU 2015 
Regular Research Award (FAPESP grant: 2015/50265-6). The authors wish to thank 
Andrzej Szczepa\'nski for helpful suggestions.

\section{Preliminaries}\label{pr}
Manifolds, mappings and tensor fields, such as bundle and covering 
projections, sub\-man\-i\-fold inclusions, and Riemannian metrics, are by 
definition of class $\,C^\infty\nnh$. Sub\-man\-i\-folds need not carry the 
subset topology, and a manifold may be disconnected (although, being required 
to satisfy the second countability axiom, it must have at most countably many 
connected components). Connectedness/compactness of a sub\-man\-i\-fold 
always refer to its own topology, and imply the same for its underlying set 
within the ambient manifold. Thus, a compact sub\-man\-i\-fold is always 
endowed with the subset topology.
By a {\smallit distribution\/} on a manifold $\,\mathcal{N}\hs$ we mean, as 
usual, a (smooth) vector sub\-bundle $\,D$ of the tangent bundle 
$\,T\mathcal{N}\nh$. An {\smallit integral manifold\/} of $\,D\,$ is any 
sub\-man\-i\-fold $\,\mathcal{L}\,$ of $\,\mathcal{N}\hs$ with 
$\,T\hskip-4pt_x\w\mathcal{L}=D\nnh_x\w$ for all $\,x\in\mathcal{L}$. The 
maximal connected integral manifolds of $\,D\,$ will also be referred to as 
the {\smallit leaves\/} of $\,D\nh$. In the case where $\,D\,$ is integrable, 
its leaves form the foliation associated with $\,D\nh$. We call $\,D\,$ 
{\smallit pro\-ject\-a\-ble\/} under a mapping 
$\,\psi:\mathcal{N}\nh\to\hat{\mathcal{N}}$ onto a distribution 
$\,\hat D\,$ on the target manifold $\,\hat{\mathcal{N}}\hs$ if 
$\,d\psi_x\w(D\nnh_x\w)=\hat D\nnh_{\psi(x)}\w$ whenever 
$\,x\in\mathcal{N}\nnh$.
\begin{remark}\label{covpr}The following well-known facts will be used below.
\begin{enumerate}
  \def\theenumi{{\rm\alph{enumi}}}
\item Free dif\-feo\-mor\-phic actions of finite groups on manifolds are 
properly discontinuous and thus give rise to covering projections onto the 
resulting quotient manifolds.
\item Any lo\-cal\-ly-dif\-feo\-mor\-phic mapping from a compact manifold into 
a connected manifold is a (surjective) finite covering projection.
\item More generally, the phrases `lo\-cal\-ly-dif\-feo\-mor\-phic mapping' 
and `finite covering projection' in (b) may be replaced with {\smallit 
submersion\/} and {\smallit fibration}.
\end{enumerate}
\end{remark}
\begin{lemma}\label{cptlv}{\smallit 
Let a distribution\/ $\,\hat D\,$ on\/ $\,\hat{\mathcal{M}}\,$ be 
pro\-ject\-a\-ble, under a locally dif\-feo\-mor\-phic surjective mapping\/ 
$\,\psi:\hat{\mathcal{M}}\to\mathcal{M}\,$ between manifolds, onto a 
distribution\/ $\,D\,$ on\/ $\,\mathcal{M}$.
\begin{enumerate}
  \def\theenumi{{\rm\roman{enumi}}}
\item The\/ $\,\psi$-im\-age of any leaf of\/ $\,\hat D\,$ is a connected 
integral manifold of\/ $\,D\nh$.
\item Integrability of\/ $\,\hat D\,$ implies that of\/ $\,D\nh$.
\item For any compact leaf\/ $\,\mathcal{L}\,$ of $\,\hat D\nh$, the image\/ 
$\,\mathcal{L}'\nh=\psi(\mathcal{L})\,$ is a compact leaf of $\,D\nh$, and 
the restriction\/ $\,\psi:\mathcal{L}\nh\to\mathcal{L}'$ constitutes 
a covering projection.
\item If the leaves of\/ $\,\hat D\,$ are all compact, so are those of\/ 
$\,D\nh$.
\end{enumerate}
}
\end{lemma}
\begin{proof}Assertion (i) is immediate from the definitions of a leaf and 
projectability, while (i) implies (ii) since integrability amounts to the 
existence of an integral manifold through every point. Remark~\ref{covpr}(b) 
combined with (i) yields (iii), and (iv) follows.
\end{proof}
\begin{lemma}\label{opdns}{\smallit 
Suppose that\/ $\,F\,$ is a mapping from a 
manifold\/ $\,\mathcal{W}$ into any set. If \,for every\/ 
$\,x\in\mathcal{W}\,$ there exists a dif\-feo\-mor\-phic identification of a 
neighborhood\/ $\,\mathcal{B}\nnh_x\w$ of\/ $\,x\,$ in\/ $\,\mathcal{W}\,$ 
with a unit open Euclidean ball centered at\/ $\,0\,$ under which\/ $\,x\,$ 
corresponds to\/ $\,0\,$ and\/ $\,F\hs$ becomes constant on each open 
straight-line interval of length\/ $\,1\,$ in the open ball having\/ $\,0\,$ 
as an endpoint, then\/ $\,F\,$ is locally constant on some open dense subset 
of\/ $\,\mathcal{W}\nh$.
}
\end{lemma}
\begin{proof}We use induction on $\,n=\dim\mathcal{W}\nh$. The case $\,n=1\,$ 
being trivial, let us assume the assertion to be valid in dimension $\,\nmo\,$ 
and consider a function $\,F\hs$ on an $\,n$-di\-men\-sion\-al manifold 
$\,\mathcal{W}\nh$, satisfying our hypothesis, along with an embedded open 
Euclidean ball 
$\,\mathcal{B}\nnh_x\w\nh\subseteq\mathcal{M}\,$ ``centered'' at a given point 
$\,x$, as in the statement of the lemma. Due to constancy of $\,F\nh$ along 
the fibres of the normalization projection 
$\,\mu:\mathcal{B}\nnh_x\w\nh\smallsetminus\{x\}\to\mathcal{S}$ onto the unit 
$\,(n\nh-\nh1)$-sphere $\,\mathcal{S}$, we may view $\,F\,$ as a mapping 
$\,G\,$ having the domain $\,\mathcal{S}$. We may now fix 
$\,y\in\mathcal{B}\nnh_x\w\nh\smallsetminus\{x\}\,$ with an embedded open 
Euclidean ball $\,\mathcal{B}\nnh_y\w$ ``centered'' at $\,y$, such that 
$\,F\hs$ is constant on each radial open interval in $\,\mathcal{B}\nnh_y\w$. 
The obvious sub\-mer\-sion property of $\,\mu\hs$ allows us to pass from 
$\,\mathcal{B}\nnh_y\w$ to a smaller concentric ball and then choose a 
co\-di\-men\-sion-one open Euclidean ball $\,\mathcal{B}'_{\!y}$ arising as a 
union of radial intervals within this smaller version of 
$\,\mathcal{B}\nnh_y\w$, for which $\,\mu:\mathcal{B}'_{\!y}\to\mathcal{S}\,$ 
is an embedding. The assumption of the lemma thus holds when 
$\,\mathcal{W}\hs$ and $\,F\hs$ are replaced by $\,\mathcal{S}\,$ and $\,G$, 
leading to local constancy of $\,G\,$ (and $\,F$) on a dense open set in 
$\,\mathcal{S}$ (and, respectively, in 
$\,\mathcal{B}\nnh_x\w\nh\smallsetminus\{x\}$). Since the union of the latter 
sets over all $\,x\,$ is obviously dense in $\,\mathcal{W}\nh$, our claim 
follows.
\end{proof}
\begin{remark}\label{nrexp}As a well-known consequence of the inverse mapping 
theorem combined with the Gauss lemma for sub\-man\-i\-folds, given a compact 
sub\-man\-i\-fold $\,\mathcal{M}\hn'$ of a Riemannian manifold 
$\,\mathcal{M}\,$ there exists $\,\rd\in(0,\infty)\,$ with the following 
properties.
\begin{enumerate}
  \def\theenumi{{\rm\alph{enumi}}}
\item The normal exponential mapping restricted to the 
radius $\,\rd\,$ open-disk sub\-bun\-dle $\,\mathcal{N}\hskip-3pt_\rd\w$ of 
the normal bundle of $\,\mathcal{M}\hn'$ constitutes a dif\-feo\-mor\-phism 
$\,\mathrm{Exp}^\perp\nnh:\mathcal{N}\hskip-3pt_\rd\w
\to\nh\mathcal{M}\nh_\rd\w$ 
onto the open sub\-man\-i\-fold $\,\mathcal{M}\nh_\rd\w$ of $\,\mathcal{M}\,$
equal to the pre\-im\-age of $\,[\hs0,\rd)\,$ under the function 
$\,\mathrm{dist}\hh(\mathcal{M}\hn'\nh,\,\cdot\,)\,$ of metric distance from 
$\,\mathcal{M}\hn'\nnh$.
\item Every $\,x\in\mathcal{M}\nh_\rd\w$ has a unique point 
$\,y\in\mathcal{M}\hn'$ nearest to $\,x$, which is simultaneously the unique 
point $\,y\,$ of $\,\mathcal{M}\hn'$ joined to $\,x\,$ by a geodesic in 
$\,\mathcal{M}\nh_\rd\w$ normal to $\,\mathcal{M}\hn'$ at $\,y$, and the 
resulting assignment $\,\mathcal{M}\nh_\rd\w\ni x\mapsto y\in\mathcal{M}\hn'$ 
coincides with the composite mapping of the inverse dif\-feo\-mor\-phism of 
$\,\mathrm{Exp}^\perp\nnh:\mathcal{N}\hskip-3pt_\rd\w
\to\nh\mathcal{M}\nh_\rd\w$ 
followed by the nor\-mal-bun\-dle projection 
$\,\mathcal{N}\hskip-3pt_\rd\w\to\nh\mathcal{M}\hn'\nnh$.
\item The $\,\mathrm{Exp}^\perp$ images of length $\,\rd\,$ radial line 
segments emanating from the zero vectors in the fibres of 
$\,\mathcal{N}\hskip-3pt_\rd\w$ coincide with the length $\,\rd\,$ minimizing 
geodesic segments in $\,\mathcal{M}\nh_\rd\w$ emanating from 
$\,\mathcal{M}\hn'\nnh$. They are all normal to the levels of 
$\,\mathrm{dist}\hh(\mathcal{M}\hn'\nh,\,\cdot\,)$, and realize the minimum 
distance between any two such levels within $\,\mathcal{M}\nh_\rd\w$.
\end{enumerate}
\end{remark}
\begin{lemma}\label{baire}{\smallit 
In a complete metric space, any countable union of closed sets with empty 
interiors has an empty interior.
}
\end{lemma}
\begin{proof}This is Baire's theorem \cite[p.\ 187]{edwards} stating, 
equivalently, that the intersection of countably many dense open subsets 
is dense.
\end{proof}

\section{Free Abel\-i\-an groups}\label{fa}
The following well-known facts, cf.\ \cite{arnold}, are gathered here for easy 
reference.

For a finitely generated Abel\-i\-an group $\,G$, being tor\-sion-free amounts 
to being free, in the sense of having a $\,\bbZ\nh$-ba\-sis, by which one 
means an ordered $\,n$-tuple $\,e_1\w,\dots,e_n\w$ of elements of $\,G\,$ such 
that every $\,x\in G\,$ can be uniquely expressed as an integer combination of 
$\,e_1\w,\dots,e_n\w$. The integer $\,n\ge0$, also denoted by $\,\dimz G$, 
is an algebraic invariant of $\,G$, called its {\smallit Bet\-ti number\/} or 
$\,\bbZ${\smallit-di\-men\-sion}.

Any finitely generated Abel\-i\-an group $\,G\,$ is isomorphic to the direct 
sum of its (necessarily finite) torsion sub\-group $\,S\,$ and the free group 
$\,G/S$, and we set $\,\dimz G=\dimz\hskip2pt[G/S]$. A sub\-group $\,G'$ (or, 
a homo\-mor\-phic image $\,G'$) of such $\,G$, in addition to being again 
finitely generated and Abel\-i\-an, also satisfies the inequality 
$\,\dimz G'\le\dimz G$, strict unless $\,G/G'$ is finite (or, 
respectively, the homo\-mor\-phism involved has a finite kernel).
\begin{lemma}\label{dirsm}{\smallit 
A sub\-group\/ $\,G'$ of finitely generated a free Abel\-i\-an group\/ $\,G\,$ 
constitutes a direct summand of\/ $\,G\,$ if and only if the quotient group\/ 
$\,G/G'$ is tor\-sion-free.
}
\end{lemma}
In fact, more generally, given a surjective homo\-mor\-phism 
$\,\chi:P\to P'$ between Abel\-i\-an groups $\,P,P'$ and elements 
$\,x\nh_j\w,y_a\w$ (with $\,j,a\,$ ranging over finite sets), such that 
$\,x\nh_j\w$ and $\,\chi(y_a\w)$ form $\,\bbZ\nh$-ba\-ses of 
$\,\mathrm{Ker}\hskip2.7pt\chi\,$ and, respectively, of $\,P'\nnh$, the system 
consisting of all $\,x\nh_j\w$ and $\,y_a\w\,$ is a $\,\bbZ\nh$-ba\-sis of 
$\,P\nh$. To see this, note that every element of $\,P\,$ then can be 
uniquely expressed as an integer combination of $\,x\nh_j\w$ and $\,y_a\w$.
\begin{lemma}\label{surjc}{\smallit 
For any finitely generated sub\-group\/ $\,G\,$ of the additive group of a 
fi\-nite-di\-men\-sion\-al real vector space\/ $\,\tvs\nnh$, the 
intersection\/ $\,G\cap\tvs\hh'$ with any vector sub\-space\/ 
$\,\tvs\hh'\subseteq\tvs\hs$ forms a di\-rect-sum\-mand sub\-group of\/ 
$\,G$. On the other hand, the class of di\-rect-sum\-mand sub\-groups of\/ 
$\,G\,$ is closed under intersections, finite or not.
}
\end{lemma}
Both claims are obvious from Lemma~\ref{dirsm}. The next lemma is a 
straightforward exercise:
\begin{lemma}\label{prdsg}{\smallit 
If normal sub\-groups\/ $\,G'\nh,G''$ of a group\/ $\,G\,$ intersect trivially 
and every\/ $\,\gamma'\in G'$ commutes with every\/ $\,\gamma''\in G''\nnh$, 
then\/ 
$\,G'G''\nh=\{\gamma'\gamma'':(\gamma'\nnh,\gamma'')\in G'\nh\times G''\}\,$ 
is a normal sub\-group\/ of\/ $\,G$, and the assignment\/ 
$\,(\gamma'\nnh,\gamma'')\mapsto\gamma'\gamma''$ defines an iso\-mor\-phism\/ 
$\,G'\nh\times G''\nh\to G'G''\nnh$.
}
\end{lemma}

\section{Lattices and vector sub\-spaces}\label{lv}
Throughout this section $\,\tvs\hs$ denotes a fixed fi\-nite-di\-men\-sion\-al 
real vector space, and
\begin{equation}\label{cpl}
\mathrm{we\ call\ sub\-spaces\ }\,\tvs\hh'\nnh,\tvs\hh''\mathrm{\ of\ 
}\,\tvs\,\text{\smallit\ complementary to each other}\,\mathrm{\ if\ 
}\,\tvs\,=\,\tvs\hh'\nh\oplus\tvs\hh''\nh.
\end{equation}
As usual, we define a (full) {\smallit lattice\/} in $\,\tvs\hs$ to be any 
sub\-group $\,L\,$ of the additive group of $\,\tvs\hs$ generated by a basis 
of $\,\tvs\hs$ (which must consequently also be a $\,\bbZ\nh$-ba\-sis of 
$\,L$).
The quotient Lie group $\,\tvs\nnh/\hskip-1ptL\,$ then is a torus, 
and we use the term {\smallit sub\-to\-ri\/} when referring to its compact 
connected Lie sub\-groups. Pro\-ject\-a\-bil\-i\-ty of distributions under 
overing projections is generally equivalent to their deck-trans\-for\-ma\-tion 
invariance; this obvious fact, applied to the projection 
$\,\tvs\nh\to\tvs\nnh/\hskip-1ptL$, 
shows that
\begin{equation}\label{tor}
\mathrm{every\ parallel\ distribution\ on\ }\,\tvs\hs\mathrm{\ is\ 
pro\-ject\-a\-ble\ onto\ the\ torus\ }\,\tvs\nnh/\hskip-1ptL.
\end{equation}
\begin{definition}\label{lsbsp}Given a lattice $\,L\,$ in $\,\tvs\nnh$, by an 
$\,L${\smallit-sub\-space\/} of $\,\tvs$ we will mean any vector sub\-space 
$\,\tvs\hh'$ of $\,\tvs\hs$ spanned by $\,L\cap\tvs\hh'\nnh$. One may 
equivalently require $\,\tvs\hh'$ to be the span of just a subset of $\,L$, 
rather than specifically of $\,L\cap\tvs\hh'\nnh$.
\end{definition}
\begin{lemma}\label{cplvs}{\smallit 
The parallel distribution on\/ $\,\tvs\hs$ tangent to 
any prescribed vector sub\-space\/ $\,\tvs\hh'$ projects onto 
a parallel distribution\/ $\,D\hs$ on the torus group\/ 
$\,\tvs\nnh/\hskip-1ptL$. The leaves of\/ $\,D\,$ must be 
either all compact, or all noncompact, and they are compact if and only if\/ 
$\,\tvs\hh'$ is an\/ $\,L$-sub\-space, in which case the leaf of\/ $\,D\,$ 
through zero is a sub\-to\-rus of\/ $\,\tvs\nnh/\nh L$.
}
\end{lemma}
\begin{proof}Pro\-ject\-a\-bil\-i\-ty is obvious from (\ref{tor}). The first 
claim about the leaves of $\,D$ follows as the leaves are one another's 
translation images. For the second, let $\,\mathcal{N}\hs$ be the leaf of 
$\,D\,$ through zero. Requiring $\,\tvs\hh'$ to be (or, not to be) an 
$\,L$-sub\-space makes $\,L\cap\tvs\hh'\nnh$, by Lemma~\ref{surjc}, a 
di\-rect-sum\-mand sub\-group of $\,L\,$ spanning $\,\tvs\hh'$ or, 
respectively, yields the existence of a nonzero linear functional $\,f\,$ on 
$\,\tvs\hh'\nnh$, the kernel of which contains $\,L\cap\tvs\hh'\nnh$. In the 
former case, $\,\mathcal{N}\hs$ is a factor of a prod\-uct-of-to\-ri 
decomposition of 
$\,\tvs\nnh/\hskip-1ptL$, while in the latter $\,f\,$ descends to 
an unbounded function on $\,\mathcal{N}\nnh$.
\end{proof}
\begin{example}\label{genrl}For a finite group $\,H\nnh$, let 
$\,\tvs\nh=\bbR\nnh^H$ and $\,L=\bbZ\nh^H\nnh$, with the convention that, 
whenever $\,X,Y\hs$ are sets, $\,Y\nnh^X\hs$ is the set of all mappings 
$\,X\hn\to Y\nnh$. Clearly,$\,\tvs\nh\cong\nh\rn$ and $\,L\cong\nh\bbZ^n$ if 
$\,H\,$ has $\,n\,$ elements, the isomorphic identifications $\,\cong\,$ 
coming from a fixed bijection $\,\{1,\dots,n\}\to H\nnh$, and so $\,L\,$ 
constitutes a lattice in the real vector space $\,\tvs\nh$. Denoting by 
$\,\tau\nh\nnh_a\w:H\to H\,$ the left translation by $\,a\in H\nnh$, we 
define a right action 
$\hs\tvs\hn\times\nh H\ni(f,a)\mapsto f\circ\tau\nh\nnh_a\w\nh\in\tvs\hs$ of 
the group $\,H\,$ on $\,\tvs\nnh$. This action -- obviously effective -- turns 
$\,H\,$ into a finite group of linear auto\-mor\-phisms of $\,\tvs\nnh$, 
preserving both $\,L\,$ and the $\,\ell\hh^2$ inner product. Any 
$\,d$-el\-e\-ment sub\-group $\,\tilde H\,$ of $\,H\,$ gives rise to two 
$\,H\nh$-in\-var\-i\-ant $\,L$-sub\-spaces 
$\,\tvs\hh'\nnh,\tvs\hh''\nnh\subseteq\tvs\nnh$, of dimensions $\,n/d\,$ and 
$\,n-n/d$, with $\,\tvs\hh'$ consisting of all $\,f\in\tvs$ constant on 
each left coset $\,a\tilde H\nnh$, $\,a\in H$, and $\,\tvs\hh''\nnh$, the 
$\,\ell\hh^2$-or\-thog\-o\-nal complement of $\,\tvs\hh'\nnh$, formed by those 
$\,f\in\tvs\hs$ having the sum of values over every coset $\,a\tilde H\,$ 
equal to zero. Thus, for the action of $\,\tilde H\,$ on $\,H\,$ via right 
translations, $\,\tvs\hh'$ (or $\,\tvs\hh''$) is the space of vectors in 
$\,\tvs\hs$ that are $\,\tilde H\nh$-in\-var\-i\-ant (or, have zero 
$\,\tilde H\nh$-av\-er\-age). Both $\,\tvs\hh'\nnh,\tvs\hh''$ are 
$\,L$-sub\-spaces: a subset of $\,L\,$ spanning $\,\tvs\hh'$ (or, 
$\,\tvs\hh''$ when $\,d>1$ and $\,\tvs\hh''\nh\ne\{0\}$) consists of functions 
equal to $\,1\,$ on one coset and to $\,0\,$ on the others (or, respectively, 
of functions assuming the values $\,1\,$ and $\,-\nnh1\,$ at two fixed points 
within the same coset, and vanishing everywhere else).
\end{example}
\begin{lemma}\label{spint}{\smallit 
Given a lattice\/ $\,L\,$ in\/ $\,\tvs\nnh$, the span and 
intersection of any family of\/ $\,L$-sub\-spaces are\/ $\,L$-sub\-spaces. The 
same is true if one replaces the phrase `$\nh L$-sub\-spaces' with\/ 
`$\nh H\nnh$-in\-var\-i\-ant $\,L$-sub\-spaces' for any fixed group\/ $\,H\hs$ 
of linear auto\-mor\-phisms of\/ $\,\tvs\hs$ sending\/ $\,L\,$ into itself.
}
\end{lemma}
\begin{proof}The assertion about spans follows from the case of two 
$\,L$-sub\-spaces, obvious in turn due to the second sentence of 
Definition~\ref{lsbsp}. Next, the intersection of the family of sub\-to\-ri in 
$\,\tvs\nnh/\hskip-1ptL$, arising via Lemma~\ref{cplvs} from the given 
family of $\,L$-sub\-spaces, constitutes a compact Lie sub\-group of 
$\,\tvs\nnh/\hskip-1ptL$, 
so that it is the union of 
finitely many cosets of a sub\-to\-rus $\,\mathcal{N}\nnh$. Since sub\-to\-ri 
are totally geodesic relative to the flat af\-fine connection on 
$\,\tvs\nnh/\hskip-1ptL$, while the projection 
$\,\tvs\nh\to\tvs\nnh/\hskip-1ptL\,$ 
locally dif\-feo\-mor\-phic, the tangent space of $\,\mathcal{N}\hs$ at 
zero equals the intersection of the tangent spaces of the sub\-to\-ri forming 
the family, and each tangent space corresponds to an $\,L$-sub\-space from our 
family. The conclusion is now immediate from Lemma~\ref{cplvs}.
\end{proof}
\begin{remark}\label{cptfd}For a lattice $\,L\,$ in $\,\tvs\hs$ generated by a 
basis $\,e_1\w,\dots,e_n\w$ of $\,\tvs\nnh$, the translational action of 
$\,L\,$ on $\,\tvs\hs$ has an obvious compact {\smallit fundamental domain\/} 
(a compact subset of $\,\tvs$ intersecting all orbits of $\,L$): the 
parallelepiped 
$\,\{t_1\w e_1\w\nh+\ldots+t_n\w e_n\w:t_1\w,\dots,t_n\w\nh\in[\hs0,1]\}$.
\end{remark}
\begin{remark}\label{lttce}We need the well-known fact \cite{charlap} that
\begin{enumerate}
  \def\theenumi{{\rm\alph{enumi}}}
\item[{\rm(a)}] lattices in $\,\tvs\hs$ are the same as discrete 
sub\-groups of $\,\tvs\nnh$, spanning $\,\tvs\nnh$.
\end{enumerate}
Given a lattice $\,L\,$ in $\,\tvs\hs$ and a vector sub\-space 
$\,\tvs\hh'\nnh\subseteq\tvs\nnh$, let $\,L\nh'\nh=L\cap\tvs\hh'\nnh$. Then
\begin{enumerate}
  \def\theenumi{{\rm\roman{enumi}}}
\item[{\rm(b)}] $L\nh'$ is a lattice in the vector sub\-space spanned by 
it, and
\item[{\rm(c)}] $L\nh'$ constitutes a di\-rect-sum\-mand sub\-group of $\,L$,
\end{enumerate}
as one sees using (a) and the first part of Lemma~\ref{surjc}.
\end{remark}
\begin{lemma}\label{ratss}{\smallit 
Let\/ $\,\tws\hs$ be the rational vector sub\-space 
of a fi\-nite-di\-men\-sion\-al real vector space\/ $\,\tvs\nnh$, spanned by 
a fixed lattice\/ $\,L\,$ in\/ $\,\tvs\nnh$. The four sets formed, 
respectively, by
\begin{enumerate}
  \def\theenumi{{\rm\roman{enumi}}}
\item $L$-sub\-spaces\/ $\,\tvs\hh'\nh$ of\/ $\,\tvs\nnh$,
\item di\-rect-sum\-mand sub\-groups\/ $\,L\nh'$ of\/ $\,L$,
\item rational vector sub\-spaces\/ $\,\tws\hh'\nh$ of\/ $\,\tws\nnh$,
\item sub\-to\-ri\/ $\,\mathcal{N}\hn'\nnh$ of the torus group\/ 
$\,\tvs\nnh/\nh L$, that is, its compact connected Lie sub\-groups,
\end{enumerate}
then stand in mutually consistent, natural bijective correspondences with one 
another, obtained by declaring\/ $\,\tvs\hh'\nh$ to be the real span of both\/ 
$\,L\nh'$ and\/ $\,\tws\hh'\nh$ as well as the identity component of the 
pre\-im\-age of\/ $\,\mathcal{N}\hn'\nh$ under\/ the projection 
homo\-mor\-phism\/ $\,\tvs\nh\to\tvs\nnh/\nh L$. Furthermore, 
$\,\tws\hh'\nh$ equals\/ $\,\tws\cap\tvs\hh'\nnh$ and, simultaneously, is the 
rational span of\/ $\,L\nh'\nnh$, while\/ 
$\,\mathcal{N}\hn'\nh=\tvs\hh'\nnh/\nh L\nh'\nh$ and\/ 
$\hs L\nh'\nh=L\cap\tvs\hh'\nh=L\cap\tws\hh'\nnh$. 
Finally, 
$\,\dimr\tvs\hh'\nh=\dimz L\nh'\nh=\dimq\tws\hh'\nh=\dim\mathcal{N}\hn'\nnh$.

`Mutual consistency' means here that the above finite set of bijections is 
closed under the operations of composition and inverse.
}
\end{lemma}
\begin{proof}The mappings (ii) $\to$ (i) and (iii) $\to$ (i), as well as (iv) 
$\to$ (i), defined in the three lines following (iv), are all bijections, 
with the inverses given by$\,(L\nh'\nnh,\tws\hh'\nnh,\mathcal{N}\hn')
=(L\cap\tvs\hh'\nnh,\tws\cap\tvs\hh'\nnh,\tvs\hh'\nnh/\nh L\nh')$. Namely, 
each of the three mappings and their purported inverses takes values in the 
correct set, and each of the six map\-ping-in\-verse compositions is the 
respective identity. To be specific, the claim about the values follows from 
Lemma~\ref{cplvs} for (iv) $\to$ (i) and (i) $\to$ (iv), from 
Definition~\ref{lsbsp} and Lemma~\ref{surjc} for (ii) $\to$ (i) and (i) 
$\to$ (ii), while it is obvious for (i) $\to$ (iii) and, for (iii) $\to$ (i), 
immediate from Definition~\ref{lsbsp}, since we are free to 
assume that
\begin{equation}\label{ide}
(L,\,\tws\nh,\,\tvs)\,\,=\,\,(\bbZ^n\nh,\,\bbQ\hn^n\nh,\,\rn)\hh,\hskip12pt
\mathrm{where\ \ }\,n=\dim\tvs\nh,
\end{equation}
and every rational vector sub\-space of $\,\bbQ\hn^n$ has a basis contained in 
$\,\bbZ^n\nnh$. Next, the compositions (ii) $\to$ (i) $\to$ (ii) and (i) $\to$ 
(ii) $\to$ (i) are the identity mappings -- the former due to the fact that 
$\,L\cap\hs\spr\,L\nh'\subseteq L\nh'$ (which one sees extending a 
$\,\bbZ\nh$-ba\-sis of $\,L\nh'$ to a $\,\bbZ\nh$-ba\-sis of $\,L$) -- 
the opposite inclusion being obvious, the latter, as Definition~\ref{lsbsp} 
gives $\,\tvs\hh'=\spr\,(L\cap\tvs\hh')$. Similarly for (iii) $\to$ (i) $\to$ 
(iii) and (i) $\to$ (iii) $\to$ (i), as long as one replaces the letters 
$\,L\,$ and $\,\bbZ\,$ with $\,\tws\hs$ and $\,\bbQ$, using (\ref{ide}) and 
the line following it. Finally, (iv) $\to$ (i) $\to$ (iv) and (i) $\to$ (iv) 
$\to$ (i) are the identity mappings as a consequence of Lemma~\ref{cplvs}, 
and the dimension equalities become obvious if one, again, chooses a 
$\,\bbZ\nh$-ba\-sis of $\,L\,$ containing a $\,\bbZ\nh$-ba\-sis of 
$\,L\nh'\nnh$.
\end{proof}
In the next theorem, as $\,H\,$ is  finite, the $\,L$-pre\-serv\-ing property 
of $\,H\,$ means that, whenever $\,A\in H$, one has 
$\,\mathrm{det}\hh A=\pm\nh1$, and so $\,AL=L\,$ (rather than just 
$\,AL\subseteq L$).
\begin{theorem}\label{invcp}{\smallit 
For a lattice\/ $\,L\,$ in a fi\-nite-di\-men\-sion\-al real vector space\/ 
$\,\tvs\nnh$, a finite group\/ $\,H$ of\/ $\,L$-pre\-serv\-ing linear 
auto\-mor\-phisms of\/ $\,\tvs\nnh$, and an\/ $\,H\nnh$-in\-var\-i\-ant\/ 
$\,L$-sub\-space\/ $\,\tvs\hh'$ of\/ $\,\tvs\nnh$, there exists an\/ 
$\,H\nnh$-in\-var\-i\-ant\/ $\,L$-sub\-space\/ $\,\tvs\hh''$ of\/ 
$\,\tvs\nnh$, complementary to\/ $\,\tvs\hh'$ in the sense of\/ 
{\rm(\ref{cpl})}.
}
\end{theorem}
\begin{proof}Let $\,\tws\hh'=\hs\tws\cap\tvs\hh'\nnh$, for the rational span 
$\,\tws\hs$ of $\,L\,$ (see Lemma~\ref{ratss}). Restricted to 
$\,\tws\nnh$, elements of $\,H\,$ act by conjugation on the rational 
af\-fine space $\,\mathcal{P}\hs$ of all $\,\bbQ$-lin\-e\-ar projections 
$\,\tws\to\tws\hh'$ (by which we mean linear operators $\,\tws\to\tws\hh'$ 
equal to the identity on $\,\tws\hh'$). The average of any orbit of the action 
of $\,H\,$ on $\,\mathcal{P}\hs$ is an $\,H\nh$-in\-var\-i\-ant 
projection $\,\tws\to\tws\hh'$ with a kernel $\,\tws\hh''$ corresponding via 
Lemma~\ref{ratss} to our required $\,\tvs\hh''\nnh$.
\end{proof}
\begin{corollary}\label{dcomp}{\smallit 
If\/ $\,L,\tvs\nnh,H\,$ satisfy the hypotheses of Theorem\/~{\rm\ref{invcp}}, 
every nonzero\/ $\,H\nnh$-in\-var\-i\-ant\/ $\,L$-sub\-space\/ 
$\,\tvs\hh'_{\hskip-2pt0}$ of\/ $\,\tvs$ can be decomposed into a direct sum 
of one or more nonzero\/ $\,H\nnh$-in\-var\-i\-ant\/ $\,L$-sub\-spaces, each 
of which is minimal in the sense of not containing any further nonzero 
proper\/ $\,H\nnh$-in\-var\-i\-ant\/ $\,L$-sub\-space.
}
\end{corollary}
\begin{proof}Induction on the possible values of 
$\,\dim\tvs\hh'_{\hskip-2pt0}$. Assuming the claim true for sub\-spaces of 
dimensions less than $\,\dim\tvs\hh'_{\hskip-2pt0}$, along with 
non-min\-i\-mal\-ity of $\,\tvs\hh'_{\hskip-2pt0}$, we fix a nonzero proper 
$\,H\nh$-in\-var\-i\-ant $\,L$-sub\-space $\,\tvs\hh'$ of $\,\tvs\nnh$, 
contained in $\,\tvs\hh'_{\hskip-2pt0}$, and choose a complement $\,\tvs\hh''$ 
of $\,\tvs\hh'\nnh$, guaranteed to exist by Theorem~\ref{invcp}. Since 
$\,\tvs\hh''$ intersects every coset of $\,\tvs\hh'$ in $\,\tvs\nnh$, 
including cosets within $\,\tvs\hh'_{\hskip-2pt0}$, the sub\-space 
$\,\tvs\hh'_{\hskip-2pt0}\cap\tvs\hh''$ is an $\,H\nh$-in\-var\-i\-ant 
complement of $\,\tvs\hh'$ in $\,\tvs\hh'_{\hskip-2pt0}$, as well as an 
$\,L$-sub\-space (due to Lemma~\ref{spint}). We may now apply the induction 
assumption to both $\,\tvs\hh'$ and 
$\,\tvs\hh'_{\hskip-2pt0}\cap\tvs\hh''\nnh$.
\end{proof}
\begin{remark}\label{quotl}Given a lattice $\,L\,$ in a 
fi\-nite-di\-men\-sion\-al real vector space $\,\tvs\hs$ and an 
$\,L$-sub\-space $\,\tvs\hh'$ of $\,\tvs\nnh$, the restriction to $\,L\,$ of 
the quo\-tient-space projection $\,\tvs\nnh\to\tvs\nh/\tvs\hh'$ has the 
kernel $\,L\nh'\nh=L\cap\tvs\hh'\nnh$, and so it descends to an injective 
group homo\-mor\-phism $\,L/\nh L\nh'\nnh\to\tvs\nh/\tvs\hh'\nnh$, the image 
of which is a (full) lattice in an $\,\tvs\nh/\tvs\hh'$ (which follows if 
one uses a $\,\bbZ\nh$-ba\-sis of $\,L\,$ containing a $\,\bbZ\nh$-ba\-sis of 
$\,L\nh'$). From now on we will treat $\,L/\nh L\nh'$ as a subset of 
$\,\tvs\nh/\tvs\hh'\nnh$. Discreteness of the lattice 
$\,L/\nh L\nh'\nnh\subseteq\tvs\nh/\tvs\hh'$ clearly implies the existence 
of an open subset $\,\mathcal{U}'$ of $\,\tvs\nnh$, containing 
$\,\tvs\hh'$ and forming a union of cosets of $\,\tvs\hh'\nnh$, such that 
$\,L\hn\cap\hs\mathcal{U}'\nh=\hs L\nh'\nnh$.
\end{remark}

\section{Af\-fine spaces}\label{as}
We denote by $\,\mathrm{End}\,\tvs\hs$ the space of linear en\-do\-mor\-phisms 
of a given real vector space $\,\tvs\nnh$. Scalars stand for the corresponding 
multiples of identity, so that the identity itself becomes 
$\,1\in\mathrm{End}\,\tvs\nnh$. For a fi\-nite-di\-men\-sion\-al real af\-fine 
space $\,\mathcal{E}\hs$ with the translation vector space $\,\tvs\nnh$, let 
$\,\mathrm{Af{}f}\,\hs\mathcal{E}\hs$ be the group of all af\-fine 
transformations of $\,\mathcal{E}\nh$. The inclusion 
$\,\mathcal{V}\hh\subseteq\mathrm{Af{}f}\,\hs\mathcal{E}\hs$ expresses the 
fact that $\,\mathrm{Af{}f}\,\hs\mathcal{E}\hs$ contains the normal sub\-group 
consisting of all translations. Any vector sub\-space $\,\tvs\hh'$ of 
$\,\tvs\hs$ gives rise to a foliation of $\,\mathcal{E}\nh$, with the leaves 
formed by af\-fine sub\-spaces $\,\mathcal{E}'$ parallel to $\,\tvs\hh'\nnh$, 
that is, orbits of the translational action of $\,\tvs\hh'$ on 
$\,\mathcal{E}\hs$ (which we may also refer to as {\smallit cosets\/} of 
$\,\tvs\hh'$ in $\,\mathcal{E}$). The resulting leaf (quotient) space 
$\,\mathcal{E}\nnh/\hn\tvs\hh'$ constitutes an af\-fine space having the 
translation vector space $\,\tvs\nnh/\hn\tvs\hh'\nnh$. Clearly,
\begin{equation}\label{opt}
\mathrm{for\ cosets\ }\,\mathcal{E}'\nnh,\mathcal{E}''\mathrm{\ of\ 
sub\-spaces\ }\,\tvs\hh'\nnh,\tvs\hh''\nnh\subseteq\tvs\hs\mathrm{\ with\ 
(\ref{cpl}),\ }\,\mathcal{E}'\nnh\cap\hs\mathcal{E}''\,\mathrm{\ is\ a\ 
one}\hyp\mathrm{point\ set.}
\end{equation}
A fixed inner product in $\,\hs\tvs\hs$ turns $\,\mathcal{E}\hs$ into a 
{\smallit Euclidean af\-fine space}, with the isom\-e\-try group 
$\,\mathrm{Iso}\,\hs\mathcal{E}\subseteq\mathrm{Af{}f}\,\hs\mathcal{E}$. If 
$\,\rd\in(0,\infty)$, we define the $\,\rd${\smallit-neigh\-bor\-hood\/} of 
an af\-fine sub\-space $\,\mathcal{E}'$ of $\,\mathcal{E}\hs$ to be the set of 
points in $\,\mathcal{E}\hs$ lying at distances less that $\,\rd\,$ from 
$\,\mathcal{E}'\nnh$. Clearly, the $\,\rd${\smallit-neigh\-bor\-hood\/} of 
$\,\mathcal{E}'$ is a union of cosets of a vector sub\-space $\,\tvs\hh'$ of 
$\,\tvs\hs$ (one of them being $\,\mathcal{E}'$ itself), as well as the 
preimage, under the projection 
$\,\mathcal{E}\nh\to\mathcal{E}\nnh/\hn\tvs\hh'\nnh$, of the radius $\,\rd\,$ 
open ball centered at the point $\,\mathcal{E}'$ in the quotient Euclidean 
af\-fine space $\,\mathcal{E}\nnh/\hn\tvs\hh'$ (for the obvious inner product 
on $\,\tvs\nnh/\hn\tvs\hh'$).
\begin{remark}\label{afext}Given a Euclidean af\-fine space 
$\,\mathcal{E}\hs$ and an af\-fine sub\-space $\,\mathcal{E}'$ parallel to a 
vector sub\-space $\,\tvs\hh'$ of the translation vector space $\,\tvs\hs$ of 
$\,\mathcal{E}\nh$, (af\-fine) self-isom\-e\-tries $\,\zeta$ of 
$\,\mathcal{E}\hs$ such that $\,\zeta(x)=x\,$ for all $\,x\in\mathcal{E}'$ are 
in an obvious one-to-one correspondence with linear self-isom\-e\-tries 
$\,A\,$ of the orthogonal complement of $\,\tvs\hh'\nnh$. In this case we will 
refer to $\,\zeta\,$ as an {\smallit af\-fine extension\/} of $\,A$, depending 
on $\,\mathcal{E}'\nnh$.
\end{remark}
\begin{remark}\label{trprt}Any choice of an origin $\,o\in\mathcal{E}\hs$ in 
an af\-fine space $\,\mathcal{E}\hs$ leads to the obvious
identification of $\,\mathcal{E}\nh$ with its translation vector space
$\,\tvs\nnh$, under which a vector $\,v\in\tvs$ corresponds to the point 
$\,x=o+v\in\mathcal{E}\nh$. Af\-fine mappings 
$\,\gamma\in\mathrm{Af{}f}\,\hs\mathcal{E}\hs$ are then represented by 
pairs $\,(A,b)\,$ consisting of $\,A\in\mathrm{End}\,\tvs\hs$ and 
$\,b\in\mathcal{E}\nh$, so that $\,\gamma(o+v)=o+Av\nh+\hn b$. The pair 
associated in this way with $\,\gamma\,$ and a {\smallit new} origin 
$\,o+w\,$ is, obviously, $\,(A,c)$, for the same $\,A$ (the linear part of 
$\,\gamma$) and $\,c=b+(A-1)w$. Thus, the coset 
$\,b+\hat{\tvs}\nh\subseteq\tvs\nnh$, where $\,\hat{\tvs}$ denotes the image 
of $\,A-1$, forms an invariant of $\,\gamma\,$ (while $\,b\,$ itself does not, 
except in the case of translations $\,\gamma$, having $\,A=1$). For any fixed 
vector subspace $\,\tvs\hh'$ of $\,\tvs\hs$ and any 
$\,\gamma\in\mathrm{Af{}f}\,\hs\mathcal{E}\hs$ with a linear part $\,A\,$ 
leaving $\,\tvs\hh'$ invariant, it now makes sense to require 
that $\,A\,$ descend to the identity transformation of 
$\,\tvs\nnh/\hn\tvs\hh'$ (i.e., $\,(A-1)(\tvs)\subseteq\tvs\hh'$) and, 
simultaneously, that the ``translational part'' $\,b\,$ of $\,\gamma\,$ lie in 
$\,\tvs\hh'\nnh$. More precisely, such a property of $\,\gamma\,$ does not 
depend on the origin used to represent $\,\gamma\,$ as a pair $\,(A,b)$.
\end{remark}
\begin{remark}\label{nrsbg}Given $\,\mathcal{E}\nh,\tvs\hs$ and $\,\tvs\hh'$ 
as in Remark~\ref{trprt}, the af\-fine transformations $\,\gamma\,$ of 
$\,\mathcal{E}$ with linear parts leaving $\,\tvs\hh'$ invariant and 
descending to the identity transformation of $\,\tvs\nnh/\hn\tvs\hh'$ 
obviously form a sub\-group of $\,\mathrm{Af{}f}\,\hs\mathcal{E}\hs$ 
containing, as a {\smallit normal sub\-group}, the set of such $\,\gamma\,$ 
which have ``translational parts'' in $\,\tvs\hh'\nnh$. This follows since the 
latter set is the kernel of the obvious homo\-mor\-phism from the original 
sub\-group into 
$\,\tvs\nnh/\hn\tvs\hh'\nh
\subseteq\nh\mathrm{Af{}f}\,[\mathcal{E}\nnh/\hn\tvs\hh']$. More 
precisely, $\,\gamma\,$ represented by the pair $\,(A,b)\,$ (see 
Remark~\ref{trprt}) preserves each element of $\,\mathcal{E}\nnh/\hn\tvs\hh'$ 
if and only if $\,Av\nh+\hn b\,$ differs from $\,v$, for every 
$\,v\in\tvs\nh$, by an element $\,\tvs\hh'$ or, equivalently (as one sees 
setting $\,v=0$), $\,\tvs\hh'$ contains both $\,b\,$ and the image of $\,A-1$.
\end{remark}
\begin{lemma}\label{fltmf}{\smallit 
Remark\/~{\rm\ref{nrexp}} has the following 
additional conclusions when\/ $\,\mathcal{M}\hn'$ is a compact leaf of a 
parallel distribution\/ $\,D\,$ on a complete flat Riemannian manifold\/ 
$\,\mathcal{M}$.
\begin{enumerate}
  \def\theenumi{{\rm\alph{enumi}}}
\item Every level of\/ $\,\mathrm{dist}\hh(\mathcal{M}\hn'\nh,\,\cdot\,)\,$ 
in\/ $\,\mathcal{M}\nh_\rd\w$, and\/ $\,\mathcal{M}\nh_\rd\w$ itself, is a 
union of leaves of\/ $\,D$.
\item Restrictions of\/ 
$\,\mathcal{M}\nh_\rd\w\ni x\mapsto y\in\mathcal{M}\hn'$ 
to leaves of\/ $\,D\,$ in\/ $\,\mathcal{M}\nh_\rd\w$ are locally isometric.
\item The local inverses of all the above lo\-cal\-ly-i\-so\-met\-ric 
restrictions correspond via the dif\-feo\-mor\-phism\/ $\,\mathrm{Exp}^\perp$ 
to all local sections of the normal bundle of\/ $\,\mathcal{M}\hn'$ obtained 
by restricting to\/ $\,\mathcal{M}\hn'$ local parallel vector fields of 
lengths\/ $\,r\in[\hs0,\rd)\,$ that are tangent to\/ $\,\mathcal{M}\,$ and 
normal to\/ $\,\mathcal{M}\hn'\nnh$, with\/ $\,\hs r\hskip.55pt$ equal to the 
value of\/ $\,\mathrm{dist}\hh(\mathcal{M}\hn'\nh,\,\cdot\,)\,$ on the leaf.
\end{enumerate}
}
\end{lemma}
This trivially follows from the fact the  pull\-back of $\,D\,$ to the 
Euclidean af\-fine space $\,\mathcal{E}$ constituting the Riemannian universal 
covering space of $\,\mathcal{M}\,$ is a distribution with the leaves provided 
by af\-fine sub\-spaces parallel to $\,\tvs\hh'\nnh$, for some vector 
sub\-space $\,\tvs\hh'$ of the translation vector space $\,\tvs\hs$ of 
$\,\mathcal{E}\nh$.

\section{Bieberbach groups and flat manifolds}\label{bg}
Let $\,\mathcal{E}\hs$ be a Euclidean af\-fine $\,n$-space (Section~\ref{as}), 
with the translation vector space $\,\tvs\nnh$. By a {\smallit Bie\-ber\-bach 
group\/} in $\,\mathcal{E}\hs$ one means any tor\-sion-free discrete 
sub\-group $\,\bg\hs$ of $\,\mathrm{Iso}\,\hs\mathcal{E}\hs$ for which there 
exists a compact fundamental domain (Remark~\ref{cptfd}). The {\smallit 
lattice sub\-group\/} $\,L$ of $\,\bg\nh$, and its {\smallit holonomy group\/} 
$\,H\subseteq\mathrm{Iso}\,\tvs\cong\mathrm{O}\hs(n)\,$ then are defined by
\begin{equation}\label{lah}
L\,=\,\bg\hn\cap\tvs,\hskip22ptH\,=\,\lp(\bg)\hh,
\end{equation}
$\lp:\mathrm{Af{}f}\,\hs\mathcal{E}\to\mathrm{Aut}\,\hn\tvs
\cong\mathrm{GL}\hs(n,\bbR)\,$ being the lin\-e\-ar-part homo\-mor\-phism. 
Thus, $\,L\,$ is the set of all translations lying in $\,\bg\hs$ (which also 
makes it the kernel of the restriction $\,\lp:\bg\to H$), and $\,H\,$ 
consists of the linear parts of elements of $\,\bg\nh$. Note that 
$\,L\subseteq\tvs\hs$ is a (full) lattice in the usual sense \cite{charlap}, 
cf.\ Section~\ref{lv}. The relations involving 
$\,\bg,L\,$ and $\,H\,$ are conveniently summarized by the short exact sequence
\begin{equation}\label{exa}
L\,\to\,\hs\bg\,\to\,H,\hskip8pt\mathrm{where\ the\ arrows\ are\ the\ 
inclusion\ homo\-mor\-phism\ and\ }\,\lp\hh.
\end{equation}
\begin{remark}\label{frpdc}The action of a Bie\-ber\-bach group $\,\bg\hs$ on 
the Euclidean af\-fine space $\,\mathcal{E}$ being always free and properly 
discontinuous, the quotient $\,\mathcal{M}=\mathcal{E}/\hn\bg\nh$, with the 
projected metric, forms a compact flat Riemannian manifold, while $\,H\,$ must 
be finite \cite{charlap}.
\end{remark}
\begin{remark}\label{cnjug}As the normal sub\-group $\,L\,$ of $\,\bg\hs$ is 
A\-bel\-i\-an, the action of $\,\bg\hs$ on $\,L$ by conjugation descends to 
an action on $\,L\,$ of the quotient group $\,\bg/L$, identified via 
(\ref{exa}) with $\,H$. This last action is clearly nothing else than the 
ordinary linear action of $\,H\,$ on $\,\tvs\nh$, restricted to the lattice 
$\,L\subseteq\tvs\,$ (and so, in particular, $\,L\,$ must be 
$\,H\nh$-in\-var\-i\-ant).
\end{remark}
\begin{remark}\label{bijct}The assignment of 
$\,\mathcal{M}=\mathcal{E}/\hn\bg\hs$ to $\,\bg\hs$ establishes a well-known 
bijective correspondence \cite{charlap} between equivalence classes of 
Bie\-ber\-bach groups and isometry types of compact flat Riemannian manifolds. 
Bie\-ber\-bach groups $\,\bg\hs$ and $\,\hat\bg\hs$ in Euclidean af\-fine 
spaces $\,\mathcal{E}\hs$ and $\,\hat{\mathcal{E}}\hs$ are called {\smallit 
equivalent\/} here if some af\-fine isometry 
$\,\mathcal{E}\to\hat{\mathcal{E}}\hs$ conjugates $\,\bg$ onto $\,\hat\bg\nh$. 
Furthermore, $\,\bg\hs$ and $\,H\,$ in (\ref{exa}) serve as the fundamental 
and holonomy groups of $\,\mathcal{M}$, while $\,\bg\hs$ also acts via deck 
tranformations on the Riemannian universal covering space of $\,\mathcal{M}$, 
isometrically identified with $\,\mathcal{E}$.
\end{remark}

\section{Lat\-tice-re\-duc\-i\-bil\-i\-ty}\label{lr}
A Bie\-ber\-bach group $\,\bg\hs$ in a Euclidean af\-fine space 
$\,\mathcal{E}\hs$ (or, the compact flat Riemannian manifold 
$\,\mathcal{M}=\mathcal{E}/\hn\bg\hs$ corresponding to $\,\bg\nh$, cf.\ 
Remark~\ref{bijct}) will be called {\smallit lat\-tice-re\-duc\-i\-ble\/} if, 
for $\,\tvs\nnh,H\,$ and $\,L\,$ associated with $\,\mathcal{E}\hs$ and 
$\,\bg\hs$ as in Section~\ref{bg}, there exists $\,\tvs\hh'$ such that
\begin{equation}\label{lrd}
\,\tvs\hh'\mathrm{\ is\ a\ nonzero\ proper\ 
}\,H\nh\hyp\mathrm{in\-var\-i\-ant\ }\,L\hyp\mathrm{subspace\ of\ 
}\,\tvs\nnh.
\end{equation}
(See Definition~\ref{lsbsp}.) To emphasize the role of $\,\tvs\hh'$ in 
(\ref{lrd}), we also say that
\begin{equation}\label{qdr}
\mathrm{the\ lat\-tice}\hyp\mathrm{re\-duc\-i\-bil\-i\-ty\ condition\ 
(\ref{lrd})\ holds\ for\ the\ quadruple\ 
}\hs(\tvs\nnh,H,L,\tvs\hh')\hh.
\end{equation}
As shown by Hiss and Szczepa\'nski 
\cite{hiss-szczepanski}, {\smallit every compact flat Riemannian manifold 
of dimension greater than one is lat\-tice-re\-duc\-i\-ble}. For more details, 
see the Appendix.

For a Bie\-ber\-bach group $\,\bg\hs$ in a Euclidean af\-fine space 
$\,\mathcal{E}\hs$ and a fixed af\-fine sub\-space $\,\mathcal{E}'$ of 
$\,\mathcal{E}\hs$ parallel to a vector sub\-space $\,\tvs\hh'$ of $\,\tvs\hs$ 
satisfying (\ref{qdr}), we denote by $\,\stb\hh'$ the {\smallit stabilizer 
sub\-group of\/} $\,\mathcal{E}'$ {\smallit relative to the action of\/} 
$\,\bg\nh$, meaning that
\begin{equation}\label{stb}
\stb\hh'\mathrm{\ consists\ of\ all\ the\ elements\ of\ }\,\bg\hs\mathrm{\ 
mapping\ }\,\mathcal{E}'\mathrm{\ onto\ itself.} 
\end{equation}
Let $\,\gamma\in\bg\nh$. From $\,\bg\nh$-in\-var\-i\-ance, cf.\ (\ref{inv}), 
of the foliation of $\,\mathcal{E}\hs$ formed by the cosets of 
$\,\tvs\hh'\nnh$,
\begin{equation}\label{iff}
\gamma\in\stb\hh'\mathrm{\ if\ and\ only\ if\ 
}\,\gamma(\mathcal{E}')\,\mathrm{\ intersects\ }\,\mathcal{E}'\nh.
\end{equation}
\begin{theorem}\label{restr}{\smallit 
Given a lat\-tice-re\-duc\-i\-ble Bie\-ber\-bach 
group\/ $\,\bg\hs$ in a Euclidean af\-fine space\/ $\,\mathcal{E}\hs$ and a 
vector subspace\/ $\,\tvs\hh'$ of\/ $\,\tvs$ with\/ {\rm(\ref{lrd})}, the 
following three conclusions hold.
\begin{enumerate}
  \def\theenumi{{\rm\roman{enumi}}}
\item The af\-fine sub\-spaces of dimension\/ $\,\dim\tvs\hh'$ in\/ 
$\,\mathcal{E}$, parallel to\/ $\,\tvs\hh'\nnh$, are the leaves of a 
foliation\/ $\,F\hskip-3pt_{\mathcal{E}}\w$ on\/ $\,\mathcal{E}\nnh$, 
pro\-ject\-a\-ble under the covering projections\/ 
$\,\mathrm{pr}:\mathcal{E}\to\mathcal{M}=\mathcal{E}/\hn\bg\hs$ and\/ 
$\,\mathcal{E}\to\mathcal{T}=\,\mathcal{E}/\nh L\,$ onto foliations\/ 
$\,F\hskip-3.8pt_{\mathcal{M}}\w\,$ of\/ $\,\mathcal{M}\,$ and\/ 
$\,F\hskip-3pt_{\mathcal{T}}\w$ of the torus\/ 
$\,\mathcal{T}=\,\mathcal{E}/\nh L$, both of which have compact totally 
geodesic leaves, tangent to a parallel distribution.
\item The leaves\/ $\,\mathcal{M}\hn'$ of\/ 
$\,F\hskip-3.8pt_{\mathcal{M}}\w\hs$ 
coincide with the\/ $\,\mathrm{pr}$-im\-ages of leaves\/ $\,\mathcal{E}'$ of\/ 
$\,F\hskip-3pt_{\mathcal{E}}\w$, and the restrictions\/ 
$\,\mathrm{pr}:\mathcal{E}'\nh\to\mathcal{M}\hn'$ are covering projections. 
The same remains true if one replaces $\,\mathcal{M}\,$ and\/ 
$\,\mathrm{pr}\,$ with\/ $\,\mathcal{T}\,$ and the projection\/ 
$\,\mathcal{E}\to\mathcal{T}\nnh$. Any such\/ $\,\mathcal{M}\hn'\nnh$, being a 
compact flat Riemannian manifold, corresponds via Remark\/~{\rm\ref{bijct}} to 
a Bie\-ber\-bach group\/ $\,\bg\hn'$ in the Euclidean af\-fine space\/ 
$\,\mathcal{E}'\nnh$. For\/ $\,L\nh'\nnh,H\nh'$ appearing in the analog\/ 
$\,L\nh'\nh\to\hs\bg\hn'\nh\to H\nh'$ of\/ {\rm(\ref{exa})}, this\/ 
$\,\bg\hn'\nnh$, and\/ $\,\stb\hh'$ defined by\/ {\rm(\ref{stb})},
\begin{enumerate}
  \def\theenumi{{\rm\alph{enumi}}}
\item $\bg\hn'$ consists of the restrictions to\/ $\,\mathcal{E}'$ of all the 
elements of\/ $\,\stb\hh'\nnh$,
\item $H\nh'$ is formed by the restrictions to\/ $\,\tvs\hh'$ of the linear 
parts of elements of \/ $\,\stb\hh'\nnh$,
\item $L\nh'\nh=\bg\hn'\nh\cap\tvs\hh'\nnh$, as in\/ {\rm(\ref{lah})}, and\/ 
$\,L\cap\tvs\hh'\subseteq L\nh'\nnh$.
\end{enumerate}
\end{enumerate}
}
\end{theorem}
We prove Theorem~\ref{restr} in the next section. 
\begin{remark}\label{stblz}The restriction homo\-mor\-phism 
$\,\stb\hh'\nnh\to\bg\hn'\nnh$, cf.\ (ii-a) above, is an {\smallit 
iso\-mor\-phism\hskip.1pt}: nontrivial elements of $\,\stb\hh'\nnh$, being 
fix\-ed-point free (Remark~\ref{frpdc}), have nontrivial restrictions to 
$\,\mathcal{E}'\nnh$. The last inclusion of (ii-c) may be proper; see the end 
of Section~\ref{gk}.
\end{remark}

\section{Proof of Theorem~\ref{restr}}\label{pt}
Projectablity of the foliation $\,F\hskip-3pt_{\mathcal{E}}\w$ under both 
covering projections 
$\,\mathrm{pr}:\mathcal{E}\to\mathcal{M}\,$ and 
$\,\mathcal{E}\to\mathcal{T}\,$ follows as a trivial consequence of the fact 
that, due to $\,H\nh$-in\-var\-i\-ance of $\,\tvs\hh'\nnh$,
\begin{equation}\label{inv}
F\hskip-3pt_{\mathcal{E}}\w\,\mathrm{\ is\ 
}\,\bg\nh\hyp\mathrm{in\-var\-i\-ant\ and, obviously,\ 
}\,L\hyp\mathrm{in\-var\-i\-ant,}
\end{equation}
while Lemma~\ref{cptlv}(ii) implies integrability of the image distribution. 
Next,
\begin{equation}\label{cmp}
\begin{array}{l}
\mathrm{pr}\nh\,\mathrm{\ is\ the\ composite\ 
}\,\nh\mathcal{E}\to\mathcal{T}\nnh\to\mathcal{M}\nh\,\mathrm{\ 
of\ two\ mappings\hskip-3pt:\hskip2.3ptthe}\\
\mathrm{u\-ni\-ver\-sal}\hyp\mathrm{cov\-er\-ing\ projection\ of\ the\ 
flat\ torus\ }\,\hs\mathcal{T}\hn=\,\mathcal{E}/\nh L,\\
\mathrm{and\ the\ quotient\ projection\ for\ the\ action\ of\ 
}\,\,\hs\bg\hs\,\,\mathrm{\ on\ }\,\,\hs\mathcal{T}\nh,
\end{array}
\end{equation}
the latter action clearly becoming free if one replaces $\,\bg\hs$ with 
$\,\bg/\nh L\cong\nh H$. Both factor mappings, 
$\,\mathcal{E}\to\mathcal{T}\,$ and 
$\,\mathcal{T}\nh\to\mathcal{M}$, are covering projections -- the first 
since $\,L\,$ is a lattice in $\,\tvs\nh$, the second due to 
Remark~\ref{covpr}(a). Parts (iii)\hs--\hs(iv) of Lemma~\ref{cptlv}, along 
with Lemma~\ref{cplvs}, may now be applied to the foliations 
$\,F\hskip-3pt_{\mathcal{T}}\w$ and $\,F\hskip-3.8pt_{\mathcal{M}}\w$ of the 
torus $\,\mathcal{T}\,$ and of $\,\mathcal{M}\,$ obtained as projections of 
$\,F\hskip-3pt_{\mathcal{E}}\w$, proving the last (com\-pact-leaves) claim of 
(i).

We now fix a leaf $\,\mathcal{E}'$ of $\,F\hskip-3pt_{\mathcal{E}}\w$, 
and choose a leaf $\,\mathcal{M}\hn'$ of $\,F\hskip-3.8pt_{\mathcal{M}}\w$ 
containing $\,\mathrm{pr}(\mathcal{E}')$, cf.\ Lemma~\ref{cptlv}(i). It 
follows that
\begin{equation}\label{cpr}
\mathrm{pr}:\mathcal{E}'\nh\to\mathcal{M}\hn'\,\mathrm{\ is\ a\ (surjective)\ 
covering\ projection,}
\end{equation}
since (\ref{cmp}) decomposes $\,\mathrm{pr}:\mathcal{E}'\nh\to\mathcal{M}\hn'$ 
into the composition 
$\,\mathcal{E}'\nh\to\mathcal{T}\hh'\nnh\to\nh\mathcal{M}\hn'\nnh$, 
in which the first mapping is the u\-ni\-ver\-sal-cov\-er\-ing projection of 
the torus $\,\mathcal{T}\hh'\nh=\hs\mathcal{E}'\nnh/\nh L\nh'\nnh$, and the 
second one must be a covering due to Remark~\ref{covpr}(b).

Two points of $\,\mathcal{E}'$ have the same $\,\mathrm{pr}$-im\-age if and 
only if one is transformed into the other by an element of the group 
$\,\bg\hn'$ described in assertion (ii). (Namely, the `only if' part follows 
since, given $\,x,y\in\mathcal{E}'$ with $\,\mathrm{pr}(x)=\mathrm{pr}(y)\,$ 
in $\,\mathcal{M}=\mathcal{E}/\hn\bg\nh$, the element of $\,\bg$ sending 
$\,x\,$ to $\,y\,$ must lie in $\,\bg\hn'\nnh$, or else, by (\ref{inv}), it 
would map $\,\mathcal{E}'$ onto a {\smallit different\/} leaf of the foliation 
$\,F\hskip-3pt_{\mathcal{E}}\w$.) Furthermore, the action of $\,\bg\hn'$ on 
$\,\mathcal{E}'$ is free due to Remark~\ref{frpdc}. Thus, $\,\bg\hn'$ 
coincides with the deck transformation group for the u\-ni\-ver\-sal 
cov\-er\-ing projection (\ref{cpr}). Now (ii) is a consequence of 
Lemma~\ref{cptlv}, Remark~\ref{bijct} and the definitions of the data 
(\ref{exa}) for any Bie\-ber\-bach group $\,\bg\nh$, applied to our 
$\,\bg\hn'\nnh$.

\section{Geometries of individual leaves}\label{gi}
Throughout this section we adopt the assumptions and notations of 
Theorem~\ref{restr}. The $\,\bg\nh$-in\-var\-i\-ance of the foliation 
$\,F\hskip-3pt_{\mathcal{E}}\w$, cf.\ (\ref{inv}), trivially gives rise to
\begin{equation}\label{act}
\mathrm{the\ obvious\ isometric\ action\ of\ }\,\bg\hs\mathrm{\ on\ the\ 
quotient\ Euclidean\ af\-fine\ space\ }\,\mathcal{E}\nnh/\hn\tvs\hh'
\end{equation}
(that is, on the leaf space of $\,F\hskip-3pt_{\mathcal{E}}\w$, the points of 
which coincide with the af\-fine sub\-spaces $\,\mathcal{E}'$ of 
$\,\mathcal{E}\hs$ parallel to $\,\tvs\hh'$). Whenever 
$\,\mathcal{E}'\nnh\in\mathcal{E}\nnh/\hn\tvs\hh'$ is fixed, $\,\stb\hh'$ in 
(\ref{stb}) obviously coincides with the iso\-tropy group of $\,\mathcal{E}'$  
for (\ref{act}). The action (\ref{act}) is not effective, as the kernel of the 
corresponding homo\-mor\-phism 
$\,\bg\to\mathrm{Iso}\,[\mathcal{E}\nnh/\hn\tvs\hh']\,$ clearly contains the 
group $\,L\nh'\nh=L\cap\tvs\hh'$ forming a lattice in $\,\tvs\hh'\nnh$, cf.\ 
Definition~\ref{lsbsp} and Remark~\ref{lttce}(b). The final clause of 
Remark~\ref{cnjug}, combined with $\,H\nh$-in\-var\-i\-ance of 
$\,\tvs\hh'\nnh$, shows that $\,L\nh'$ is a normal sub\-group of $\,\bg\nh$, 
which leads to a further homo\-mor\-phism 
$\,\bg/\nh L\nh'\nh\to\mathrm{Iso}\,[\mathcal{E}\nnh/\hn\tvs\hh']\,$ (still in 
general noninjective; see Remark~\ref{nninj} below). Let 
$\,\mathrm{pr}\,$ again stand for the covering projection 
$\,\mathcal{E}\to\mathcal{M}=\mathcal{E}/\hn\bg\nh$.
\begin{remark}\label{uncvr}Given 
$\,\mathcal{E}'\nnh\in\mathcal{E}\nnh/\hn\tvs\hh'$ and a vector 
$\,v\in\tvs\hs$ orthogonal to $\,\tvs\hh'\nnh$, let us set 
$\,\mathcal{M}_{\hn v}'\nh=\mathrm{pr}(\mathcal{E}'\nnh+v)$, so that 
$\,\mathcal{M}_{\hn0}'=\mathcal{M}\hn'\nnh$, cf.\ (\ref{cpr}). By (\ref{cpr}), 
$\,\mathcal{M}_{\hn v}'$ must be a (compact) leaf of 
$\,F\hskip-3.8pt_{\mathcal{M}}\w$, and 
$\,\mathrm{pr}:\mathcal{E}'\nnh+v\to\mathcal{M}_{\hn v}'$ is a 
lo\-cal\-ly-i\-so\-met\-ric u\-ni\-ver\-sal-cov\-er\-ing projection. Also, we
\begin{equation}\label{svp}
\begin{array}{l}
\mathrm{choose\ }\,\hs\rd\hs\,\mathrm{\ as\ in\ Remark~\ref{nrexp}\ and\ 
Lemma~\ref{fltmf}\ for\ the\ sub\-man\-i\-fold}\\
\mathcal{M}\hn'\nnh=\mathrm{pr}(\mathcal{E}')\mathrm{,\ 
cf.\ (\ref{cpr}),\ of\ the\ compact\ flat\ manifold\ 
}\,\mathcal{M}=\mathcal{E}/\hn\bg\nnh,\\
\mathrm{and\ denote\ by\ }\,\stb_{\nnh v}'\hh\subseteq\nh\bg\,\mathrm{\ the\ 
iso\-tropy\ group\ of\ }\,\mathcal{E}'\nnh+v\mathrm{,\ as\ in\ (\ref{stb}).}
\end{array} 
\end{equation}
\end{remark}
\begin{lemma}\label{isolv}{\smallit 
Under the above hypotheses, for any\/ 
$\,\mathcal{E}'\nnh\in\mathcal{E}\nnh/\hn\tvs\hh'$ there exists\/ 
$\,\rd\in(0,\infty)$ such that, whenever\/ $\,u\in\tvs\hs$ is a unit vector 
orthogonal to\/ $\,\tvs\hh'$ and\/ $\,r,s\in(0,\rd)$, the iso\-me\-tries\/ 
$\,\mathcal{E}'\nnh+ru\to\mathcal{E}'\nnh+su\,$ and 
$\,\mathcal{E}'\nnh+ru\to\mathcal{E}'$ of translations by the vectors 
$\,(s-r)u\,$ and, respectively, $\,-ru\hh$, descend under the 
u\-ni\-ver\-sal-cov\-er\-ing projections of Remark\/~{\rm\ref{uncvr}}, with\/ 
$\,v\,$ equal to\/ $\,ru,su\,$ or\/ $\,0$, to an iso\-me\-try\/ 
$\,\mathcal{M}_{\hn ru}'\nh\to\mathcal{M}_{\hn su}'$ or, respectively, a\/ 
$\,k$-fold covering projection\/ 
$\,\mathcal{M}_{\hn ru}'\nh\to\mathcal{M}\hn'\nnh$, where the integer\/ 
$\,k=k(u)\ge1\,$ may depend on $\,u$, but not on\/ $\,r$.
}
\end{lemma}
\begin{proof}For $\,\rd\,$ selected in (\ref{svp}) and any 
$\,c\in[\hs0,1]$, let 
$\,\psi\nnh_c\w:\mathcal{M}\nh_\rd\w\to\nh\mathcal{M}\nh_\rd\w$ correspond, 
under the $\,\mathrm{Exp}^\perp\nnh$-dif\-feo\-mor\-phic identification of 
Remark~\ref{nrexp}(a), to the mapping 
$\,\mathcal{N}\hskip-3pt_\rd\w\to\nh\mathcal{N}\hskip-3pt_\rd\w$ which 
multiplies vectors normal to $\,\mathcal{M}\hn'$ by the scalar $\,c$. With 
$\,\phi\,$ denoting our iso\-me\-try 
$\,\mathcal{E}'\nnh+ru\to\mathcal{E}'\nnh+su\,$ (or, 
$\,\mathcal{E}'\nnh+ru\to\mathcal{E}'$) we now have 
$\,\mathrm{pr}\circ\phi=\psi\nnh_c\w\nh\circ\hh\mathrm{pr}\,$ on 
$\,\mathcal{E}'\nnh+ru$, where $\,c=s/r\,$ 
(or, respectively, $\,c=0$) \,since, given $\,x\in\mathcal{E}'\nnh$, the 
$\,\mathrm{pr}$-im\-age of the line segment $\,\{x+tu:\in[\hs0,\rd)\}$ in 
$\,\mathcal{E}\hs$ is the length $\,\rd\,$ minimizing geodesic segment in 
$\,\mathcal{M}\nh_\rd\w$ emanating from the point 
$\,y=\mathrm{pr}(x)\in\mathcal{M}\hn'$ in a direction normal to 
$\,\mathcal{M}\hn'\nnh$, and $\,\mathrm{pr}\circ\phi\,$ sends $\,x+tu$, in 
both cases, to $\,\mathrm{pr}(x+ctu)=\psi\nnh_c\w(\mathrm{pr}(x+tu))$. The 
$\,\mathrm{pr}$-im\-age of $\,\phi(z)$, for any $\,z\in\mathcal{E}'\nnh+ru$, 
thus depends only on $\,\mathrm{pr}(z)\,$ (by being its 
$\,\psi\nnh_c\w\nh$-im\-age), and so both original isometries $\,\phi\,$ 
descend to (necessarily lo\-cal\-ly-i\-so\-met\-ric) mappings 
$\,\mathcal{M}_{\hn ru}'\nh\to\mathcal{M}_{\hn su}'$ and 
$\,\mathcal{M}_{\hn ru}'\nh\to\mathcal{M}\hn'\nnh$, which constitute finite 
coverings (Remark~\ref{covpr}(b)). The former is also bi\-jec\-tive, 
its inverse arising when one switches $\,r\hs$ and $\,s$. As the composite 
$\,\mathcal{M}_{\hn su}'\nh\to\mathcal{M}_{\hn ru}'\nh\to\mathcal{M}\hn'$ 
clearly equals the analogous covering projection 
$\,\mathcal{M}_{\hn su}'\nh\to\mathcal{M}\hn'$ (with $\,s\,$ rather than 
$\,r$), the coverings $\,\mathcal{M}_{\hn ru}'\nh\to\mathcal{M}\hn'$ and 
$\,\mathcal{M}_{\hn su}'\nh\to\mathcal{M}\hn'$ have the same multiplicity, 
which completes the proof.
\end{proof}
\begin{remark}\label{delta}Replacing $\,\rd\,$ of (\ref{svp}) with $\,1/4\,$ 
times its original value, we can also require it to have the following 
property: {\smallit if\/ $\,\gamma\in\bg\hs$ and\/ $\hs\,x\in\mathcal{E}\hs$ 
are such that both\/ $\,x$ and\/ $\,\gamma(x)\,$ lie in the\/ 
$\,\rd$-neigh\-bor\-hood of\/ $\,\mathcal{E}'\nnh$, cf.\ 
Section\/~{\rm\ref{as}}, then\/ $\,\gamma\in\stb\hh'$ for the stabilizer 
group\/ $\,\stb\hh'$ of\/ $\,\mathcal{E}'$ defined by\/ {\rm(\ref{stb})}}. 
In fact, letting $\,\mathcal{E}''$ be the leaf of 
$\,F\hskip-3pt_{\mathcal{E}}\w$ through $\,x$, we see from (\ref{inv}) that 
its $\,\gamma$-im\-age $\,\gamma(\mathcal{E}'')\,$ is also a leaf of 
$\,F\hskip-3pt_{\mathcal{E}}\w$, while both leaves are within the distance 
$\,\rd\,$ from $\,\mathcal{E}'\nnh$, which gives 
$\,\mathrm{dist}\hh(\mathcal{E}''\nnh,\gamma(\mathcal{E}''))<2\rd\,$ and 
so, due to the triangle inequality, 
$\,\mathrm{dist}\hh(\mathcal{E}'\nnh,\gamma(\mathcal{E}'))
\le\mathrm{dist}\hh(\mathcal{E}'\nnh,\mathcal{E}'')
+\mathrm{dist}\hh(\mathcal{E}''\nnh,\gamma(\mathcal{E}''))
+\mathrm{dist}\hh(\gamma(\mathcal{E}''),\gamma(\mathcal{E}'))
<\rd+2\rd+\rd=4\rd$. Thus, $\,x+ru\in\gamma(\mathcal{E}')\,$ for some 
$\,x\in\mathcal{E}'\nnh$, some unit vector $\,u\in\tvs\hs$ orthogonal to 
$\,\tvs\hh'\nnh$, and 
$\,r=\mathrm{dist}\hh(\mathcal{E}'\nnh,\gamma(\mathcal{E}'))\in[\hs0,4\rd)$. 
Assuming now (\ref{svp}) with $\,\rd\,$ replaced by $\,4\rd$, one gets 
$\,r=0$, that is, $\,\gamma(\mathcal{E}')=\mathcal{E}'$ and 
$\,\gamma\in\stb\hh'\nnh$. Namely, the $\,\mathrm{pr}$-im\-age of the curve 
$\hs[\hs0,4\rd)\ni t\mapsto x+tu\hs$ is a geodesic in the image of the 
dif\-feo\-mor\-phism $\,\mathrm{Exp}^\perp$ of Remark~\ref{nrexp}(a), which 
intersects $\,\mathcal{M}\hn'$ only at $\,t=0$, while 
$\,\mathcal{M}\hn'\nh=\mathrm{pr}(\mathcal{E}')
=\mathrm{pr}(\gamma(\mathcal{E}'))$, since 
$\,\mathcal{M}=\mathcal{E}/\hn\bg\nh$.
\end{remark}
\begin{lemma}\label{cmmut}{\smallit 
Let there be given\/ $\,\tvs\hh'\nnh,\mathcal{E}'$ as in 
Lemma\/~{\rm\ref{isolv}}, $\,\rd\,$ having the additional property of 
Remark\/~{\rm\ref{delta}}, any\/ $\,r\in(0,\rd)$, and any unit vector\/ 
$\,u\in\tvs$ orthogonal to\/ $\,\tvs\hh'\nnh$.
\begin{enumerate}
  \def\theenumi{{\rm\alph{enumi}}}
\item The iso\-tropy group\/ $\,\stb_{\hskip-1ptr\hn u}'$, cf.\ 
{\rm(\ref{svp})}, does not depend on\/ $\,r\in(0,\rd)$.
\item The linear part of each element of\/ $\,\stb_{\hskip-1ptr\hn u}'$ 
keeps\/ $\,u\,$ fixed.
\item $\stb_{\hskip-1ptr\hn u}'$ is a sub\-group of\/ $\,\stb_{\nh0}'$ 
with the finite index\/ $\,k=k(u)\ge1\,$ of \,Lemma\/~{\rm\ref{isolv}},
\item $\mathrm{pr}:\mathcal{E}\to\mathcal{M}\,$ 
maps the\/ $\,\rd$-neigh\-bor\-hood\/ $\,\mathcal{E}\nnh_\rd\w$ of\/ 
$\,\mathcal{E}'$ in\/ $\,\mathcal{E}\hs$ onto\/ $\,\mathcal{M}\nh_\rd\w\hs$ of 
\,Remark\/~{\rm\ref{nrexp}(a)}.
\item $\mathcal{E}\nnh_\rd\w$ \ and\/ $\,\mathcal{M}\nh_\rd\w\,$ are unions of 
leaves of, respectively, $\,F\hskip-3pt_{\mathcal{E}}\w\,$ and\/ 
$\,F\hskip-3.8pt_{\mathcal{M}}\w$.
\item[{\rm(\hs f\hskip1.4pt)}] The pre\-im\-age under\/ 
$\,\mathrm{pr}:\mathcal{E}\nnh_\rd\w\nh\to\mathcal{M}\nh_\rd\w$ of the leaf\/ 
$\,\mathcal{M}_{\hn ru}'=\hs\mathrm{pr}(\mathcal{E}'\nnh+ru)\,$ of\/ 
$\,F\hskip-3.8pt_{\mathcal{M}}\w\,$ equals the union of the images\/ 
$\,\gamma(\mathcal{E}'\nnh+ru)\,$ over all\/ $\,\gamma\in\stb_{\nh0}'$.
\end{enumerate}
}
\end{lemma}
\begin{proof}By (\ref{cpr}) and (\ref{svp}), 
$\,\mathcal{M}_{\hn v}'=(\mathcal{E}'\nnh+v)\hs/\hn\bg_{\nnh v}'$, if one lets 
$\,\bg_{\nnh v}'$ denote the image of $\,\stb_{\nnh v}'$ under the 
injective homo\-mor\-phism of restriction to $\,\mathcal{E}'\nnh$, cf.\ 
Remark~\ref{stblz}. Fixing $\,s\in[\hs0,\rd)$ and $\,r\in(0,\rd)\,$ we 
therefore conclude from Lemma~\ref{isolv} and (\ref{iff}) that, whenever 
$\,x\in\mathcal{E}'\nnh+ru$ and $\,\gamma\in\stb_{\hskip-1ptr\hn u}'$, 
there exists $\,\hat\gamma\in\stb_{\nnh su}'$ satisfying the condition
\begin{equation}\label{gxp}
\gamma(x)\,+\,v\,=\,\hat\gamma(x+v)\hh,\hskip12pt\mathrm{where\ 
}\,\,v=(s-r)u\hh,\hskip8pt\mathrm{and\ }\,\,\hat\gamma=\gamma\,\,\mathrm{\ 
when\ }\,\,s=r,
\end{equation}
the last clause being obvious since $\,\gamma,\hat\gamma\in\bg\hs$ and the 
action of $\,\bg\hs$ is free. With $\,u\,$ and $\,\gamma\,$ fixed as well, for 
each given $\,\hat\gamma\in\stb_{\nnh su}'$ the set of all 
$\,x\in\mathcal{E}'\nnh+ru\,$ having the property (\ref{gxp}) is closed in 
$\,\mathcal{E}'\nnh+ru\,$ while, as we just saw, the union of these sets over 
all $\,\hat\gamma\in\stb_{\nnh su}'$ equals $\,\mathcal{E}'\nnh+ru$. Thus, 
by Baire's theorem (Lemma~\ref{baire}), some 
$\,\hat\gamma\in\stb_{\nnh su}'$ 
satisfies (\ref{gxp}) with all $\,x\,$ from some nonempty open subset of 
$\,\mathcal{E}'\nnh+ru$, and hence -- by real-an\-a\-lyt\-ic\-i\-ty -- for 
all $\,x\in\mathcal{E}'\nnh+ru$. In terms of the translation $\,\tau\nnh_v\w$ 
by the vector $\,v$, we consequently have 
$\,\hat\gamma=\tau\nnh_v\w\circ\gamma\circ\tau\nnh_v^{-\nh1}$ on 
$\,\mathcal{E}'\nnh+su$, so that $\,\gamma\,$ uniquely determines 
$\,\hat\gamma\,$ (due to the injectivity claim of Remark~\ref{stblz}), the 
assignment $\,\gamma\mapsto\hat\gamma\,$ is a homo\-mor\-phism 
$\,\stb_{\hskip-1ptr\hn u}'\to\stb_{\nnh su}'\nh\subseteq\hh\bg\nh$, and 
$\,\zeta
=\hat\gamma\circ\tau\nnh_v\w\circ\gamma^{-\nh1}\nh\circ\tau\nnh_v^{-\nh1}$ 
equals the identity on $\,\mathcal{E}'\nnh+su$. If we now allow $\,s\,$ to 
vary from $\,r$ to $\,0$, the resulting curve $\,s\mapsto\zeta$ consists, 
due to Remark~\ref{afext}, of af\-fine extensions of linear 
self-isom\-e\-tries of the orthogonal complement of $\,\tvs\hh'\nnh$, and 
$\,\hat\gamma=\zeta\circ\tau\nnh_v\w\circ\gamma\circ\tau\nnh_v^{-\nh1}$ on 
$\,\mathcal{E}\nh$. As $\,\bg\hs$ is discrete, the curve 
$\,s\mapsto\hat\gamma\in\bg\nh$, with $\,v=(s-r)u$, must be constant, and can 
be evaluated by setting $\,s=r\,$ (or, $\,v=0$). Thus, 
$\,\hat\gamma=\gamma\,$ on $\,\mathcal{E}\hs$ from the last clause of 
(\ref{gxp}), and so 
$\,\stb_{\hskip-1ptr\hn u}'\nh\subseteq\hh\stb_{\nnh su}'$. For $\,s>0$, 
switching $\,r$ with $\,s\,$ we get the opposite inclusion, and (a) follows. 
Also, taking the linear part of the resulting relation 
$\,\gamma=\zeta\circ\tau\nnh_v\w\circ\gamma\circ\tau\nnh_v^{-\nh1}$, we see 
that $\,\zeta\,$ equals the identity, for all $\,s$. Hence 
$\,\gamma=\tau\nnh_v\w\circ\gamma\circ\tau\nnh_v^{-\nh1}$ commutes with 
$\,\tau\nnh_v\w$, which amounts to (b). Setting 
$\,s=0$, we obtain the first part of (c): 
$\,\stb_{\hskip-1ptr\hn u}'\nh\subseteq\hh\stb_{\nh0}'$. Assertion (d) is 
clear as $\,\mathrm{pr}$, being locally isometric, maps line segments onto 
geodesic segments. Lemma~\ref{fltmf}(a) for 
$\,D=F\hskip-3.8pt_{\mathcal{M}}\w$ yields (e). Since 
$\,\mathrm{pr}:\mathcal{E}\to\mathcal{M}=\mathcal{E}/\hn\bg\nh$, the 
additional property of $\,\rd$ (see Remark~\ref{delta}) implies (f). Finally, 
for $\,k=k(u)$, the geodesic $\,[\hs0,r]\ni t\mapsto\mathrm{pr}(x+tu)$, normal 
to $\,\mathcal{M}\hn'$ at $\,y=\mathrm{pr}(x)$, is one of $\,k\,$ such 
geodesics $\,[\hs0,r]\ni t\mapsto\mathrm{pr}(x+tw)$, joining $\,y\,$ to points 
of its pre\-im\-age under the 
projection $\,\mathcal{M}_{\hn ru}'\nh\to\mathcal{M}\hn'$ of 
Lemma~\ref{isolv}, where $\,w$ ranges over a $\,k$-el\-e\-ment set 
$\,\mathcal{R}\,$ of unit vectors in $\,\tvs\nh$, orthogonal to 
$\,\tvs\hh'\nnh$. The union of the corresponding subset 
$\,C=\{\mathcal{E}'\nnh+rw:w\in\mathcal{R}\}$ of the leaf space of 
$\,F\hskip-3pt_{\mathcal{E}}\w$ equals the pre\-im\-age in (f) -- and hence an 
orbit for the action of $\,\stb_{\nh0}'$ -- as every leaf in the 
pre\-im\-age contains a point nearest $\,x$. Due to the 
al\-read\-y-es\-tab\-lished inclusion 
$\,\stb_{\hskip-1ptr\hn u}'\nh\subseteq\hh\stb_{\nh0}'$ and (\ref{svp}), 
$\,\stb_{\hskip-1ptr\hn u}'$ is the isotropy group of 
$\,\mathcal{E}'\nnh+ru\,$ relative to the transitive action of 
$\,\stb_{\nh0}'$ on $\,C$, and so $\,k$, the cardinality of $\,C$, equals 
the index of $\,\stb_{\hskip-1ptr\hn u}'$ in $\,\stb_{\nh0}'$, which 
proves the second part of (c).
\end{proof}

\section{The generic iso\-tropy group}\label{gg}
Given a Bie\-ber\-bach group $\,\bg\hs$ in a Euclidean af\-fine space 
$\,\mathcal{E}\hs$ with the translation vector space $\,\tvs\nnh$, let us fix 
a vector subspace $\,\tvs\hh'$ of $\,\tvs\hs$ satisfying (\ref{lrd}). As long 
as $\,\dim\,\mathcal{E}\ge2$, such $\,\tvs\hh'$ always exists 
(Section~\ref{lr}). An element $\,\mathcal{E}'$ of 
$\,\mathcal{E}\nnh/\hn\tvs\hh'\nnh$, that is, a coset of $\,\tvs\hh'$ in 
$\,\mathcal{E}\nh$, will be called {\smallit generic\/} if its stabilizer 
(iso\-tropy) sub\-group $\,\stb\hh'\subseteq\bg\nh$, defined by (\ref{stb}), 
equals
\begin{equation}\label{krh}
\mathrm{the\ kernel\ of\ the\ homo\-mor\-phism\ 
}\,\bg\to\mathrm{Iso}\,[\mathcal{E}\nnh/\hn\tvs\hh']\mathrm{\ corresponding\ 
to\ (\ref{act}).}
\end{equation}
The $\,\mathrm{pr}$-im\-ages of generic cosets 
of $\,\tvs\hh'$ will be called generic leaves 
of $\,F\hskip-3.8pt_{\mathcal{M}}\w$.

Still using the symbols $\,\mathrm{pr},L\,$ and $\,H\,$ for the 
u\-ni\-ver\-sal-cov\-er\-ing projection 
$\,\mathcal{E}\to\mathcal{M}=\mathcal{E}/\hn\bg$ and the groups appearing in 
(\ref{lah}) -- (\ref{exa}), let us also
\begin{equation}\label{kpr}
\begin{array}{l}
\mathrm{denote\ by\ }\,\,K'\subseteq H\,\,\mathrm{\ the\ normal\ sub\-group\ 
consisting\ of\ all\ elements}\\
\mathrm{of\ }\,\,H\hs\,\mathrm{\ that\ act\ on\ the\ orthogonal\ complement\ 
of\ }\,\tvs\hh'\hn\mathrm{\ as\ the\ identity\nh,}\\
\mathrm{and\ by\ }\hs\mathcal{U}'\nnh\mathrm{\ the\ subset\ of\ 
}\hs\mathcal{E}\nnh/\hn\tvs\hh'\nnh\mathrm{\ formed\ by\ all\ generic\ cosets\ 
of\ }\hs\tvs\hh'\nnh\mathrm{\ in\ }\hs\mathcal{E}.
\end{array} 
\end{equation}
\begin{theorem}\label{gnric}{\smallit 
For\/ $\,\stb\hh'$ equal to\/ {\rm(\ref{krh})}, our assumptions yield the 
following conclusions.
\begin{enumerate}
  \def\theenumi{{\rm\roman{enumi}}}
\item $\mathcal{U}'$ constitutes an open dense subset of\/ 
$\,\mathcal{E}\nnh/\hn\tvs\hh'\nnh$.
\item The normal sub\-group\/ $\,\stb\hh'$ of\/ $\hs\bg\hs$ is contained as a 
fi\-nite-in\-dex sub\-group in the iso\-tropy group of every\/ 
$\,\mathcal{E}'\nnh\in\mathcal{E}\nnh/\hn\tvs\hh'$ for the action\/ 
{\rm(\ref{act})}, and equal to this iso\-tropy group if\/ 
$\,\mathcal{E}'\nnh\in\mathcal{U}'\nnh$.
\item The\/ $\,\mathrm{pr}$-im\-ages\/ 
$\,\mathcal{M}\hn'\nnh,\mathcal{M}\hn''$ of any\/ 
$\,\mathcal{E}'\nnh,\mathcal{E}''\nnh\in\mathcal{U}'$ are isometric to each 
other.
\item If one identifies\/ $\,\mathcal{E}\hs$ with its translation vector space 
$\,\tvs$ via a choice of an origin, $\,\stb\hh'\nh$ becomes the set of all 
elements of\/ $\,\bg\hs$ having, for\/ $\,K\hn'$ given by\/ {\rm(\ref{kpr})}, 
the form
\begin{equation}\label{axb}
\tvs\nh\ni\hn x\hn\mapsto\nh Ax\nh+\hn 
b\in\tvs\nnh\mathrm{,\ in\ which\ }\,b\nh\in\nnh\tvs\hh'\nnh\mathrm{\ and\ 
the\ linear\ part\ }\hs A\,\mathrm{\ lies\ in\ }\hs K\hn'.
\end{equation}
\item Whenever\/ $\,\mathcal{E}'\nnh\in\mathcal{U}'\nnh$, the homo\-mor\-phism 
which restricts elements of the generic iso\-tropy group\/ $\,\stb\hh'\nh$ 
to\/ $\,\mathcal{E}'$ is injective, and the resulting iso\-mor\-phic image\/ 
$\,\bg\hn'$ of\/ $\,\stb\hh'\nh$ constitutes a Bie\-ber\-bach group in the 
Euclidean af\-fine space\/ $\,\mathcal{E}'\nnh$. The lattice sub\-group of\/ 
$\,\bg\hn'$ and its holonomy group\/ $\,H\nh'$ are the intersection\/ 
$\,L\nh'\nh=L\cap\tvs\hh'$ and the image\/ $\,H\nh'$ of the group\/ $\,K\hn'$ 
defined in\/ {\rm(\ref{kpr})} under the injective homo\-mor\-phism of 
restriction to\/ $\,\tvs\hh'\nnh$.
\end{enumerate}
}
\end{theorem}
\begin{proof}Lemma~\ref{cmmut}(a) states that the assumptions of 
Lemma~\ref{opdns} are satisfied by the Euclidean af\-fine space 
$\,\mathcal{W}\nnh=\mathcal{E}\nnh/\hn\tvs\hh'$ and the mapping $\,F\,$ that 
sends $\,\mathcal{E}'\nnh\in\mathcal{E}\nnh/\hn\tvs\hh'$ to its iso\-tropy 
group $\,\stb\hh'$ with (\ref{stb}). The assignment 
$\,\mathcal{E}'\nh\mapsto\stb\hh'$ is thus locally constant on some open dense 
set $\,\mathcal{U}'\nh\subseteq\hs\mathcal{E}\nnh/\hn\tvs\hh'\nnh$. Letting 
$\,\stb\hh'$ be the constant value of this assignment assumed on a nonempty 
connected open subset $\,\mathcal{W}'$ of $\,\mathcal{U}'\nnh$, and fixing 
$\,\gamma\in\stb\hh'\nnh$, we obtain $\,\gamma(\mathcal{E}')=\mathcal{E}'$ for 
all $\,\mathcal{E}'\nnh\in\mathcal{W}'\nnh$, and hence, from 
real-an\-a\-lyt\-ic\-i\-ty, for all 
$\,\mathcal{E}'\nnh\in\mathcal{E}\nnh/\hn\tvs\hh'\nnh$. Thus, our 
$\,\stb\hh'$ is contained in the iso\-tropy group of every 
$\,\mathcal{E}'\nnh\in\mathcal{E}\nnh/\hn\tvs\hh'\nnh$. Since the same applies 
also to another constant value $\,\stb\hh''$ assumed on a nonempty connected 
open set, $\,\stb\hh''\nnh=\stb\hh'$ and the phrase `locally constant' may be 
replaced with {\smallit constant}. By Lemma~\ref{cmmut}(c), such $\,\stb\hh'$ 
must be a fi\-nite-in\-dex sub\-group of each iso\-tropy group. As $\,\stb\hh'$ 
consists of the elements of $\,\bg\hs$ preserving every 
$\,\mathcal{E}'\nnh\in\mathcal{U}'\nnh$, real-an\-a\-lyt\-ic\-i\-ty implies 
that they preserve all $\,\mathcal{E}'\nnh\in\mathcal{E}\nnh/\hn\tvs\hh'\nnh$, 
and so $\,\stb\hh'$ coincides with (\ref{krh}), which also shows that 
$\,\stb\hh'$ is a normal sub\-group of $\,\bg\nh$, and (i) -- (ii) follow.

Assertions (iv) -- (v) are in turn immediate from Remark~\ref{nrsbg} and, 
respectively, Theorem~\ref{restr}(ii) combined with Remark~\ref{stblz}, while 
(v) implies (iii) via Remark~\ref{bijct}.
\end{proof}
\begin{remark}\label{nninj}An element of $\,\bg\hs$ acting trivially on 
$\,\mathcal{E}\nnh/\hn\tvs\hh'$ need not lie in $\,L\nh'\nnh$. An example 
arises when the compact flat manifold $\,\mathcal{M}=\mathcal{E}/\hn\bg\hs$ is 
a Riemannian product 
$\,\mathcal{M}\nh=\nnh\mathcal{M}\hn'\nh\times\hn\mathcal{M}\hn''$ 
with $\,\mathcal{E}\nh=\mathcal{E}'\nnh\times\mathcal{E}''$ and 
$\,\bg=\bg\hn'\nnh\times\bg\hn''$ for two Bie\-ber\-bach groups 
$\,\bg\hn'\nh,\bg\hn''$ in Euclidean af\-fine spaces 
$\,\mathcal{E}'\nh,\mathcal{E}''$ having the translation vector spaces 
$\,\tvs\hh'\nnh,\tvs\hh''\nnh$, while $\,\mathcal{M}\hn'$ is not a torus. The 
$\,H\nh$-in\-var\-i\-ant sub\-space $\,\tvs\hh'\nh\times\nh\{0\}\,$ then gives 
rise to the $\,\mathcal{M}\hn'$ factor foliation 
$\,F\hskip-3pt_{\mathcal{E}}\w$ of the product manifold $\,\mathcal{M}$, and 
the action of the group $\,\bg\hn'\nh\times\nh\{1\}\,$ on its leaf space is 
obviously trivial, even though $\,\bg\hn'\nh\times\nh\{1\}\,$ contains some 
elements that are not translations.
\end{remark}
\begin{remark}\label{subgp}In Theorem~\ref{gnric}, if 
$\,\mathcal{E}'\nnh\in\mathcal{U}'\nnh$, we may treat $\,\bg\hn'$ (or, 
$\,H\nh'$) as a sub\-group of $\,\bg\hs$ (or, respectively, $\,H$), by 
identifying $\,\stb\hh'$ with $\,\bg\hn'$ (or, $\,H\nh'$ with $\,K\hn'$) via 
the iso\-mor\-phism $\,\stb\hh'\nh\to\bg\hn'$ and $\,K\hn'\nh\to H\nh'$ 
resulting from Theorem~\ref{gnric}(v). Note that these sub\-groups 
$\,\bg\hn'\nh\subseteq\bg\hs$ and $\,H\nh'\nh\subseteq H\,$ do not 
depend on the choice of $\,\mathcal{E}'\nnh\in\mathcal{U}'\nnh$, and neither 
does the mapping degree $\,d=|H\nh'|\,$ of the $\,d\hh$-fold covering 
projection  $\,\mathcal{T}\hh'\nnh\to\nh\mathcal{M}\hn'\nh
=\mathcal{T}\hh'\nnh/\nh H\nh'\nnh$, cf.\ (\ref{cmp}) and the line following 
it.
\end{remark}
\begin{remark}\label{spcas}Any lattice $\,L\,$ in the translation vector space 
$\,\tvs\hs$ of a Euclidean af\-fine space $\,\mathcal{E}\hs$ is, obviously, 
a Bie\-ber\-bach group in $\,\mathcal{E}\nh$. In the case of a fixed vector 
subspace $\,\tvs\hh'$ of $\,\tvs\hs$ with (\ref{lrd}), all general facts 
established about any given Bie\-ber\-bach group $\,\bg\hs$ in 
$\,\mathcal{E}\nh$, the compact flat manifold 
$\,\mathcal{M}=\mathcal{E}/\hn\bg\nh$, and the leaves $\,\mathcal{M}\hn'$ of 
$\,F\hskip-3.8pt_{\mathcal{M}}\w\hs$ (see Theorem~\ref{restr}) thus remain 
valid for the torus $\,\mathcal{T}=\,\mathcal{E}/\nh L\,$ and the leaves 
$\,\mathcal{T}\hh'$ of $\,F\hskip-3pt_{\mathcal{T}}\w$. Every coset of 
$\,\tvs\hh'$ becomes generic if we declare the lattice $\,L\,$ of $\,\bg\hs$ 
to be the new Bie\-ber\-bach group.
\end{remark}

\section{The leaf space}\label{ls}
By a {\smallit crystallographic group\/} \cite{szczepanski} in a Euclidean 
af\-fine space one means a discrete group of isometries having a compact 
fundamental domain, cf.\ Remark~\ref{cptfd}.
\begin{proposition}\label{cryst}{\smallit 
Under the assumptions listed at the beginning of Section\/~{\rm\ref{gg}}, 
with\/ $\,\stb\hh'$ denoting the normal sub\-group\/ {\rm(\ref{krh})} of\/ 
$\,\bg\nh$, the quotient group\/ $\,\bg/\hn\stb\hh'$ acts effectively by 
isometries on the quotient Euclidean af\-fine space\/ 
$\,\mathcal{E}\nnh/\hn\tvs\hh'$ and, when identified a sub\-group of\/ 
$\,\mathrm{Iso}\,[\mathcal{E}\nnh/\hn\tvs\hh']$, it constitutes a 
crystallographic group.
}
\end{proposition}
\begin{proof}A compact fundamental domain exists by Remark~\ref{cptfd} since, 
according to Remark~\ref{quotl}, $\,\bg/\hn\stb\hh'$ contains the lattice 
sub\-group $\,L/\nh L\nh'$ of $\,\mathcal{E}\nnh/\hn\tvs\hh'\nnh$. To verify 
discreteness of $\,\bg/\hn\stb\hh'\nnh$, suppose that, on the contrary, some 
sequence $\,\gamma_k\w\in\bg\hs$, $\,k=1,2\,\dots\hs$, has terms representing 
mutually distinct elements of $\,\bg/\hn\stb\hh'$ which converge in 
$\,\mathrm{Iso}\,[\mathcal{E}\nnh/\hn\tvs\hh']$. As $\,L\nh'$ is a 
lattice in $\,\tvs\hh'\nnh$, fixing $\,x\in\mathcal{E}\hs$ and suitably 
choosing $\,v_k\w\in L\nh'$ we achieve boundedness of the sequence 
$\,\hat\gamma_k\w(x)=\gamma_k\w(x)+v_k\w$, while $\,\hat\gamma_k\w$ represent 
the same (distinct) elements of $\,\bg/\hn\stb\hh'$ as $\,\gamma_k\w$. The 
ensuing  convergence of a sub\-se\-quence of $\,\hat\gamma_k\w$ contradicts 
discreteness of $\,\bg\nh$.
\end{proof}
The resulting quotient of $\,\mathcal{E}\nnh/\hn\tvs\hh'$ under the action of 
$\,\bg/\hn\stb\hh'$ is thus a flat compact or\-bi\-fold \cite{davis}, which 
may clearly be identified both with the leaf space 
$\,\mathcal{M}/\nh F\hskip-3.8pt_{\mathcal{M}}\w$, and with the quotient of 
the torus $\,[\mathcal{E}\nnh/\hn\tvs\hh']/[L\cap\tvs\hh']\,$ under the action 
of $\,H$, mentioned in (\ref{lsp}). The latter identification clearly implies 
the Haus\-dorff property the leaf space $\,\mathcal{M}/\nh F\hskip-3.8pt_{\mathcal{M}}\w$.

On the other hand, for an $\,H\nh$-in\-var\-i\-ant sub\-space $\,\tvs\hh''$ 
of $\,\tvs\hs$ {\smallit not\/} assumed to be an $\,L$-sub\-space, there 
exists an $\,L${\smallit-clo\-sure\/} of $\,\tvs\hh''\nnh$, meaning the 
smallest $\,L$-sub\-space $\,\tvs\hh'$ of $\,\tvs\hs$ containing 
$\,\tvs\hh''\nnh$, which is obviously obtained by intersecting all such 
$\,L$-sub\-spaces (Lemma~\ref{spint}). The leaf space 
$\,\mathcal{M}/\nh F\hskip-3.8pt_{\mathcal{M}}\w$ corresponding to 
$\,\tvs\hh'$ then forms a natural ``Haus\-dorff\-iz\-ation'' of the leaf space 
of $\,\tvs\hh''\nnh$, and may also be described in terms of 
Haus\-dorff-Gro\-mov limits. See the forthcoming paper 
\cite{bettiol-derdzinski-mossa-piccione}.

\section{Intersections of generic complementary leaves}\label{ig}
For a ho\-mol\-o\-gy interpretation of parts (a) and (c) in 
Theorem~\ref{intrs}, see Remark~\ref{inthm}.

Throughout this section $\,\bg\hs$ is a given Bie\-ber\-bach group in a 
Euclidean af\-fine space $\,\mathcal{E}\hs$ of dimension $\,n\ge2$, while 
$\,\tvs\hh'\nnh,\tvs\hh''$ are two mutually complementary 
$\,H\nnh$-in\-var\-i\-ant $\,L$-sub\-spaces of the translation vector space 
$\,\tvs\hs$ of $\,\mathcal{E}\nh$, in the sense of (\ref{cpl}) and 
Definition~\ref{lsbsp}, for $\,L\,$ and $\,H\,$ associated with $\,\bg\hs$ via 
(\ref{lah}). We also fix generic cosets $\,\mathcal{E}'$ of $\,\tvs\hh'$ 
and $\,\mathcal{E}''$ of $\,\tvs\hh''$ (see the beginning of 
Section~\ref{gg}), which leads to the analogs 
$\,L\nh'\nh\to\hs\bg\hn'\nh\to H\nh'$ and 
$\,L\nh''\nh\to\hs\bg\hn''\nh\to H\nh''$ of (\ref{exa}), described by 
Theorem~\ref{restr}(ii) and, $\,\mathcal{E}'\nnh,\mathcal{E}''$ being generic, 
Theorem~\ref{gnric}(v) yields $\,L\nh'\nh=L\cap\tvs\hh'$ and 
$\,L\nh''\nh=L\cap\tvs\hh''\nnh$. Furthermore, for these 
$\,\bg\nh,\bg\hn'\nnh,\bg\hn''\nnh,L,L\nh'\nnh,L\nh''\nnh,
H,H\nh'\nnh,H\nh''\nnh$,
\begin{equation}\label{cnc}
\mathrm{the\ conclusions \ of\ Lemma~\ref{prdsg}\ hold\ if\ we\ replace\ the\ 
letter\ }\,G\,\mathrm{\ with\ }\,\bg\nh,L\,\mathrm{\ or\ }\,H,
\end{equation}
since so do the assumptions of Lemma~\ref{prdsg}, provided that one uses 
Remark~\ref{subgp} to treat $\,\bg\hn'$ and $\,\bg\hn''$ (or, $\,H\nh'$ and 
$\,H\nh''$) as sub\-groups of $\,\bg\hs$ (or, respectively, $\,H$). In fact, 
(\ref{axb}) and the description of $\,K\hn'$ in (\ref{kpr}) show that all 
$\,A\in K\hn'$ (and, among them, the linear parts of all elements of 
$\,\stb\hh'\nh=\bg\hn'$) leave invariant both $\,\tvs\hh'$ and 
$\,\tvs\hh''\nnh$, and act via the identity on the latter. (We have the 
obvious isomorphic identifications of $\,\tvs\nnh/\hn\tvs\hh'$ with 
$\,\tvs\hh''$ on the one hand, and with the orthogonal complement of 
$\,\tvs\hh'$ in $\,\tvs\hs$ on the other, while such $\,A\,$ descend to the 
identity auto\-mor\-phism of $\,\tvs\nnh/\hn\tvs\hh'\nnh$.) The same is 
clearly the case if one switches the primed symbols with the double-prim\-ed 
ones, while elements of $\,\stb\hh'$ now commute with those of $\,\stb\hh''$ 
in view of (\ref{axb}). This yields (\ref{cnc}), so that we may form
\begin{equation}\label{qgr}
\mathrm{the\ quotient\ groups\ }\,\hat\bg\nh
=\bg/(\bg\hn'\nnh\bg\hn'')\hh,\hskip8pt\hat L
=L/(L\nh'\nnh L\nh'')\hh,\hskip8pt\hat H\nh=H/(H\nh'\nnh H\nh'')\hh.
\end{equation}
Finally, let $\,\mathrm{pr}:\mathcal{E}\to\mathcal{M}=\mathcal{E}/\hn\bg\hs$ 
and 
$\,\mathcal{M}\hn'\nnh,\mathcal{M}\hn''\nnh,\mathcal{T}\hh'\nnh,
\mathcal{T}\hh''$ denote, respectively, the covering projection of 
Theorem~\ref{restr}(i), the $\,\mathrm{pr}$-im\-age $\,\mathcal{M}\hn'$ of 
$\,\mathcal{E}'$ (or, $\,\mathcal{M}\hn''$ of $\,\mathcal{E}''$), and the tori 
$\,\mathcal{E}'\nnh/\nh L\nh'$ and $\,\mathcal{E}''\nnh/\nh L\nh''\nnh$, 
contained in the torus $\,\mathcal{T}\nh=\hs\mathcal{E}/\nh L\,$ of 
(\ref{cmp}). Note that $\,\mathcal{M}\hn'$ and $\,\mathcal{M}\hn''$ are 
(compact) leaves of the parallel distributions arising, due to 
Theorem~\ref{restr}(i), on $\,\mathcal{M}=\mathcal{E}/\hn\bg\nh$, which 
itself is a compact flat Riemannian manifold (Remark~\ref{frpdc}).
\begin{theorem}\label{intrs}{\smallit 
Under the above hypotheses, the following conclusions hold.
\begin{enumerate}
  \def\theenumi{{\rm\alph{enumi}}}
\item The intersection\/ $\,\mathcal{M}\hn'\cap\mathcal{M}\hn''\nnh$, or\/ 
$\,\mathcal{T}\hh'\cap\mathcal{T}\hh''\nnh$, is a finite subset of\/ 
$\,\mathcal{M}$, or $\,\mathcal{T}\nnh$, and stands in a bijective 
correspondence with the quotient group\/ $\,\hat \bg\hs$ or, 
respectively, $\,\hat L$, of\/ {\rm(\ref{qgr})},
\item The projection\/ $\,\mathcal{T}\nnh\to\mathcal{M}\,$ of\/ 
{\rm(\ref{cmp})} maps\/ $\,\mathcal{T}\hh'\cap\mathcal{T}\hh''$ 
injectively into\/ $\,\mathcal{M}\hn'\cap\mathcal{M}\hn''\nnh$.
\item The cardinality\/ $\,|\mathcal{M}\hn'\cap\mathcal{M}\hn''\hn|\,$ of\/ 
$\,\mathcal{M}\hn'\cap\mathcal{M}\hn''$ equals\/ 
$\,|\mathcal{T}\hh'\cap\mathcal{T}\hh''\hn|\,$ times\/ 
$\,|\hat H|$.
\item The claim about\/ $\,\mathcal{T}\hh'\cap\mathcal{T}\hh''$ in\/ {\rm(a)} 
remains true whether or not\/ $\,\mathcal{E}'\nnh,\mathcal{E}''$ are generic.
\item The two bijective correspondences in\/ {\rm(a)} may be chosen so that, 
under the resulting identifications, the injective mapping\/ 
$\,\mathrm{pr}:\mathcal{T}\hh'\cap\mathcal{T}\hh''\nh
\to\mathcal{M}\hn'\cap\mathcal{M}\hn''$ of\/ {\rm(b)} coincides with the 
group homo\-mor\-phism\/ 
$\,\hat L\to\hat \bg\hs$ induced by the inclusion\/ $\,L\to\bg\nh$.
\end{enumerate}
}
\end{theorem}
\begin{proof}We first prove (a) for 
$\,\mathcal{M}\hn'\cap\mathcal{M}\hn''\nnh$. Finiteness of 
$\,\mathcal{M}\hn'\cap\mathcal{M}\hn''$ follows as 
$\,\mathcal{M}\hn'\nnh,\mathcal{M}\hn''\nnh$, and hence also 
$\,\mathcal{M}\hn'\cap\mathcal{M}\hn''\nnh$, are compact totally geodesic 
sub\-man\-i\-folds of $\,\mathcal{M}$, while 
$\,\mathcal{M}\hn'\cap\mathcal{M}\hn''\nnh$, nonempty by (\ref{opt}), has 
$\,\dim(\mathcal{M}\hn'\cap\mathcal{M}\hn'')=0\,$ due to (\ref{cpl}). The 
mapping $\,\varPsi:\bg\to\mathcal{M}\hn'\cap\mathcal{M}\hn''$ with 
$\,\mathrm{pr}(\mathcal{E}'\nh\cap\gamma(\mathcal{E}''))
=\{\varPsi(\gamma)\}\,$ is 
well defined in view of (\ref{cpl}) applied to $\,\gamma(\mathcal{E}'')\,$ 
rather than $\,\mathcal{E}''\nnh$, and clearly takes values in both 
$\,\mathcal{M}\hn'\nh=\mathrm{pr}(\mathcal{E}')\,$ and 
$\,\mathcal{M}\hn''\nh=\mathrm{pr}(\mathcal{E}'')
=\mathrm{pr}(\gamma(\mathcal{E}''))$. Surjectivity of $\,\varPsi\hs$ follows: 
if $\,\mathrm{pr}(x'')\in\mathcal{M}\hn'\cap\mathcal{M}\hn''\nnh$, where 
$\,x''\nh\in\mathcal{E}''$ then, obviously, 
$\,\mathrm{pr}(x'')=\mathrm{pr}(x')\,$ and $\,x'\nh=\gamma(x'')\,$ for some 
$\,x'\nh\in\mathcal{E}'$ and $\,\gamma\in\bg\nh$, so that 
$\,x'\nh\in\mathcal{E}'\nh\cap\gamma(\mathcal{E}'')\,$ and 
$\,\mathrm{pr}(x'')=\mathrm{pr}(x')\,$ equals $\,\varPsi(\gamma)$, the unique 
element of $\,\mathrm{pr}(\mathcal{E}'\nh\cap\gamma(\mathcal{E}''))$. 
Furthermore, $\,\varPsi$-pre\-im\-ages of elements of 
$\,\mathcal{M}\hn'\cap\mathcal{M}\hn''$ are precisely the cosets of the normal 
sub\-group $\,\bg\hn'\nnh\bg\hn''\nnh$ of $\,\bg\hs$ (which clearly implies (a) 
for $\,\mathcal{M}\hn'\cap\mathcal{M}\hn''$). Namely, the left and right 
cosets coincide, and so elements $\,\gamma\nh_1\w,\gamma\nh_2\w$ of 
$\,\bg\hs$ lie in the same coset of $\,\bg\hn'\nnh\bg\hn''$ if and only if 
\begin{equation}\label{smc}
\gamma'\nh\circ\gamma\nh_1\w=\gamma\nh_2\w\circ\gamma''\mathrm{\ \ for\ 
some\ }\,\gamma'\nh\in\bg\hn'\mathrm{\ and\ }\,\gamma''\nh\in\bg\hn''\nh.
\end{equation}
Now let $\,\gamma\nh_1\w,\gamma\nh_2\w$ lie in the same coset of 
$\,\bg\hn'\nnh\bg\hn''\nnh$. For $\,\gamma'\nh,\gamma''$ with (\ref{smc}), 
$\,\gamma'(\mathcal{E}')=\mathcal{E}'$ and 
$\,\gamma''(\mathcal{E}'')=\mathcal{E}''$ by the definition (\ref{stb}) of 
$\,\stb\hh'\nnh,\stb\hh''$ and their identification with 
$\,\bg\hn'\nnh,\bg\hn''$ (see above). Thus, 
$\,\{\varPsi(\gamma\nh_1\w)\}
=\mathrm{pr}(\mathcal{E}'\nh\cap\gamma\nh_1\w(\mathcal{E}''))
=\mathrm{pr}(\gamma'(\mathcal{E}'\nh\cap\gamma\nh_1\w(\mathcal{E}'')))
=\mathrm{pr}(\gamma'(\mathcal{E}')\cap\gamma'(\gamma\nh_1\w(\mathcal{E}'')))
=\mathrm{pr}(\mathcal{E}'\nh\cap\gamma'(\gamma\nh_1\w(\mathcal{E}'')))
=\mathrm{pr}(\mathcal{E}'\nh\cap\gamma\nh_2\w(\gamma''(\mathcal{E}'')))
=\mathrm{pr}(\mathcal{E}'\nh\cap\gamma\nh_2\w(\mathcal{E}''))
=\{\varPsi(\gamma\nh_2\w)\}$. Conversely, if 
$\,\gamma\nh_1\w,\gamma\nh_2\w\in\bg\hs$ and 
$\,\varPsi(\gamma\nh_1\w)=\varPsi(\gamma\nh_2\w)$, the unique points 
$\,x\hn_1\w$ of $\,\mathcal{E}'\nh\cap\gamma\nh_1\w(\mathcal{E}'')\,$ 
and $\,x\hn_2\w$ of $\,\mathcal{E}'\nh\cap\gamma\nh_2\w(\mathcal{E}'')$ 
both lie in the same $\,\bg$-or\-bit, and hence 
$\,x\hn_2\w=\gamma(x\hn_1\w)\,$ with some $\,\gamma\in\bg\nh$. For 
$\,\gamma'\nh=\gamma$ and 
$\,\gamma''\nh=\gamma\nh_2^{-\nh1}\nh\circ\gamma\circ\gamma\nh_1\w$, the image 
$\,\gamma'(\mathcal{E}')\,$ (or, $\,\gamma''(\mathcal{E}'')$) intersects 
$\,\mathcal{E}'$ (or, $\,\mathcal{E}''$), the common point being 
$\,x\hn_2\w=\gamma(x\hn_1\w)$ or, respectively, 
$\,\gamma\nh_2^{-\nh1}(x\hn_2\w)=\gamma\nh_2^{-\nh1}(\gamma(x\hn_1\w))$. From 
(\ref{iff}) we thus obtain $\,\gamma'\nh\in\stb\hh'\nh=\bg\hn'$ and 
$\,\gamma''\nh\in\stb\hh''\nh=\bg\hn''\nnh$, which yields (\ref{smc}).

Assertion (a) for $\,\mathcal{T}\hh'\cap\mathcal{T}\hh''\nnh$, along with (d), 
now follows as a special case; see Remark~\ref{spcas}.

Except for the word `injective' the claim made in (e) is immediate if one uses 
the mapping $\,\varPsi:\bg\to\mathcal{M}\hn'\cap\mathcal{M}\hn''$ defined 
above and its analog $\,L\to\mathcal{T}\hh'\cap\mathcal{T}\hh''$ obtained by 
replacing $\,\bg\nh,\mathcal{M}\hn'\nnh,\mathcal{M}\hn''$ and 
$\,\mathrm{pr}\,$ with $\,L,\mathcal{T}\hh'\nnh,\mathcal{T}\hh''$ and the 
projection $\,\mathcal{E}\to\mathcal{T}\nh=\hs\mathcal{E}/\nh L$. This yields 
(b), injectivity of the homo\-mor\-phism 
$\,\hat L\to\hat \bg\hs$ being immediate: if an element of $\,L\,$ lies in 
$\,\bg\hn'\nnh\bg\hn''$ (and hence has the form 
$\,\gamma'\nh\circ\gamma''\nnh$, where 
$\,(\gamma'\nnh,\gamma'')\in\bg\hn'\nh\times\bg\hn''$), (\ref{axb}) implies 
that $\,\gamma'\nnh,\gamma''$ are translations with 
$\,\gamma'\nh\in L\nh'\nh=L\cap\tvs\hh'$ and 
$\,\gamma''\nh\in L\nh''\nh=L\cap\tvs\hh''$ (see the lines preceding
(\ref{cnc})); in other words, $\,\gamma'\nh\circ\gamma''$ represents zero in 
$\,\hat L$.

Finally, $\,\hat L\,$ identified as above with a sub\-group of $\,\hat \bg\hs$ 
equals the kernel of the clear\-ly-sur\-jec\-tive homo\-mor\-phism 
$\,\hat\bg\to\hat H\nh$, induced by $\,\bg\to H\,$ in (\ref{exa}) (which, 
combined with (e), proves (c)). Namely, $\,\hat L\,$ contains the kernel (the 
other inclusion being obvious): if the linear part of $\,\gamma\in\bg\hs$ lies 
in $\,H\nh'\nnh H\nh''\nnh$, and so equals the linear part of 
$\,\gamma'\nh\circ\gamma''$ for some 
$\,(\gamma'\nnh,\gamma'')\in\bg\hn'\nh\times\bg\hn''\nnh$, then 
$\,\gamma=\lambda\circ\gamma'\nh\circ\gamma''\nnh$, where $\,\lambda\in L$, as 
required.
\end{proof}

\section{Leaves and integral ho\-mol\-o\-gy}\label{ih}
This section once again employs the assumptions and notations of 
Theorem~\ref{restr}, with $\,\dim\tvs\nh=n\,$ and $\,\dim\tvs\hh'\nh=k$, where 
$\,0<k<n$. As the holonomy group 
$\,H\subseteq\mathrm{Iso}\,\tvs\cong\mathrm{O}\hs(n)$ is finite 
(Remark~\ref{frpdc}), $\,\mathrm{det}\hh(H)\subseteq\{1,-\nh1\}$. 
In other words, the elements of $\,H\,$ have the determinants $\,\pm\nh1$. 
Using the covering projection 
$\,\mathcal{T}\nnh\to\mathcal{M}=\mathcal{T}\nnh/\nh H\nh$, cf.\ 
(\ref{cmp}) and the line following it, we see that
\begin{equation}\label{ori}
\mathrm{the\ condition\ }\,\mathrm{det}\hh(H)=\{1\}\,\mathrm{\ amounts\ to\ 
o\-ri\-ent\-a\-bil\-i\-ty\ of\ }\,\mathcal{M}\hh.
\end{equation}
By Theorem~\ref{gnric}(iii), the generic leaves 
of $\,F\hskip-3.8pt_{\mathcal{M}}\w$, defined as in the line following 
(\ref{krh}), are either all o\-ri\-ent\-a\-ble or all non\-o\-ri\-ent\-a\-ble.
\begin{lemma}\label{orien}{\smallit Let\/ $\,\mathcal{M}\,$ be 
o\-ri\-ent\-a\-ble. Then all the generic leaves\/ $\,\mathcal{M}\hn'$ of\/ 
$\,F\hskip-3.8pt_{\mathcal{M}}\w\hs$ may be o\-ri\-ent\-ed so as to represent 
the same nonzero\/ $\,k$-di\-men\-sion\-al real ho\-mol\-o\-gy class\/ 
$\,[\mathcal{M}\hn']\in H\nnh_k\w(\mathcal{M},\bbR)$.
}
\end{lemma}
\begin{proof}A fixed orientation of $\,\tvs\hh'\nnh$, being preserved, due to 
(\ref{kpr}) -- (\ref{axb}) and (\ref{ori}), by the generic iso\-tropy group 
$\,\stb\hh'\nnh$, gives rise to orientations of all leaves 
$\,\mathcal{T}\hh'$ of $\,F\hskip-3pt_{\mathcal{T}}\w$ and all generic leaves 
$\,\mathcal{M}\hn'$ of $\,F\hskip-3.8pt_{\mathcal{M}}\w$, so as to make 
the covering projections $\,\mathcal{T}\hh'\nnh\to\nh\mathcal{M}\hn'$ in the 
line following (\ref{cpr}) o\-ri\-en\-ta\-tion-pre\-serv\-ing. Since the torus 
group $\,\tvs\nnh/\nh L\,$ acts transitively on the o\-ri\-ent\-ed leaves 
$\,\mathcal{T}\hh'\nnh$, they all represent a single real ho\-mol\-o\-gy class 
$\,[\mathcal{T}\hh']\in H\nnh_k\w(\mathcal{T}\nh,\bbR)$, equal to the image of 
the fundamental class of $\,\mathcal{T}\hh'$ under the inclusion 
$\,\mathcal{T}\hh'\nnh\to\mathcal{T}\nnh$. At the same time, for generic 
leaves $\,\mathcal{M}\hn'\nnh$, the $\,d\hh$-fold covering projection 
$\,\mathcal{T}\hh'\nnh\to\nh\mathcal{M}\hn'$ (where $\,d=|H\nh'|$ does not 
depend on the choice of $\,\mathcal{M}\hn'\nnh$, cf.\ Remark~\ref{subgp}) 
sends the fundamental class of $\,\mathcal{T}\hh'$ to $\,d\,$ times the 
fundamental class of $\,\mathcal{M}\hn'\nnh$. Thus, by functoriality, 
$\,d[\mathcal{M}\hn']\in H\nnh_k\w(\mathcal{M},\bbR)$ is the image of 
$\,[\mathcal{T}\hh']\in H\nnh_k\w(\mathcal{T}\nh,\bbR)\,$ under the covering 
projection $\,\mathcal{T}\nnh\to\mathcal{M}$, which makes it the same for all 
generic leaves $\,\mathcal{M}\hn'\nnh$. Finally, $\,[\mathcal{M}\hn']\ne0$, 
since a fixed constant positive differential $\,k$-form on the o\-ri\-ent\-ed 
space $\,\tvs\hh'$ descends, in view of the first line of this proof, to a 
parallel positive volume form on each o\-ri\-ent\-ed generic leaf 
$\,\mathcal{M}\hn'\nnh$, which yields a positive value when integrated over 
$\,[\mathcal{M}\hn']$.
\end{proof}
\begin{remark}\label{inthm}If $\,\mathcal{M}\,$ is o\-ri\-ent\-a\-ble, the 
first two cardinalities in Theorem~\ref{intrs}(c) equal the intersection 
numbers of the real ho\-mol\-o\-gy classes 
$\,[\mathcal{M}\hn'],[\mathcal{M}\hn'']$, or 
$\,[\mathcal{T}\hh'],[\mathcal{T}\hh'']$, arising via Lemma~\ref{orien}, which 
is consistent with the fact that -- according to Theorem~\ref{intrs}(a) 
and Remark~\ref{subgp} -- they depend just on the two mutually complementary 
$\,H\nnh$-in\-var\-i\-ant $\,L$-sub\-spaces $\,\tvs\hh'\nnh,\tvs\hh''$ of 
$\,\tvs\nnh$, and not on the individual generic leaves 
$\,\mathcal{M}\hn'\nnh,\mathcal{M}\hn''\nnh,\mathcal{T}\hh'$ or 
$\,\mathcal{T}\hh''\nnh$.
\end{remark}

\section{Generalized Klein bottles}\label{gk}
This section presents some known examples \cite[p.\ 163]{charlap} to 
illustrate our discussion.

Let $\,\varSigma\,$ and $\,r\nnh_\theta\w:\varSigma\to\varSigma\,$ denote the 
unit circle in $\,\bbC\,$ and the rotation by angle $\,\theta$ 
(multiplication by $\,e^{i\theta}$). For 
$\,(t,\psi)\in\bbR\times\bbZ\nh^\varSigma$ and $\,f\in\bbR\nh^\varSigma\nnh$, 
cf.\ Example~\ref{genrl}, the assignment
\begin{equation}\label{tpf}
((t,\psi),f)\,\mapsto\,f\circ r\nnh_{2\pi t}\w\hs+\,t\,+\,\psi\hh,
\end{equation}
defines a left 
action on $\,\bbR\nh^\varSigma$ of the group 
$\,\bbR\times\bbZ\nh^\varSigma\nnh$, with the group operation 
$\,(t,\psi)(t'\nh,\psi')=(t\,+\,t',\,\psi'\nh\circ r\nnh_{2\pi t}\w\nh+\psi)$. 
The term $\,t\,$ in (\ref{tpf}) is the constant function 
$\,t:\varSigma\to\bbR$, and one has the obvious short exact sequence 
$\,\bbZ\nh^\varSigma\nh\to\hs\bbR\times\bbZ\nh^\varSigma\nh\to\hs\bbR$, the 
arrows being $\,\psi\mapsto(0,\psi)$ and, respectively, 
$\,(t,\psi)\mapsto t$. 

The functions $\,f:\varSigma\to\bbR\,$ are not assumed continuous and, 
whenever $\,H\subseteq\varSigma$, we treat $\,\bbR\nnh^H$ (and $\,\bbZ\nh^H$) 
as subsets of $\,\bbR\nh^\varSigma$ (and $\,\bbZ\nh^\varSigma$) via the zero 
extension of functions $\hs H\to\bbR\hs$ to $\,\varSigma$. For $\,n\ge2\,$ and 
the group $\,H=\bbZ_n\w\subseteq\varSigma\,$ of $\,n$th roots of unity, 
$\,\bbZ\nh^H\cong\bbZ^n$ is a lattice in the Euclidean space 
$\,\tvs\nh=\bbR\nnh^H\cong\rn\nnh$, and the action (\ref{tpf}) has a 
restriction to an af\-fine isometric action of the subgroup 
$\,\bg=[(\nh1\nh/\nh n\nh)\bbZ]\times\bbZ_0^{\hskip-1.1ptH}
\subseteq\bbR\times\bbZ\nh^\varSigma$ on $\,\tvs\nh$, with 
the subgroup $\,\bbZ_0^{\hskip-1.1ptH}\cong\bbZ^\nmo$ of $\,\bbZ\nh^H$ given 
by $\,\{\psi\in\bbZ\nh^H:\psi\nnh_{\mathrm{avg}}\w\nh=0\}$, where 
$\,(\hskip2.7pt)\nh_{\mathrm{avg}}\w$ denotes the averaging functional 
$\,\tvs\nh\to\bbR$. Note that, in the right-hand side of (\ref{tpf}) for 
$\,(t,\psi)\in\bg\nh$,
\begin{equation}\label{avg}
t_{\mathrm{avg}}\w\hs=\,t\hh,\hskip28pt
\psi\nnh_{\mathrm{avg}}\w\hs=\,0\hh,\hskip28pt
(f\nnh\circ\nnh r\nnh_{2\pi t}\w)\nh_{\mathrm{avg}}\w\hs
=\,f\hskip-2.3pt_{\mathrm{avg}}\w\hh.
\end{equation}
\begin{lemma}\label{biebg}
{\smallit
These\/ $\,H\nh,\tvs\,$ and\/ $\,\bg\hs$ have the following properties.
\begin{enumerate}
  \def\theenumi{{\rm\roman{enumi}}}
\item The action of\/ $\,\bg\hs$ on\/ $\,\tvs\,$ is effective.
\item $\bg\hs$ constitutes a Bie\-ber\-bach group in the underlying Euclidean 
af\-fine\/ $\,n$-space of\/ $\,\tvs\nh$.
\item The holonomy group and lattice sub\-group of\/ $\,\bg\hs$ are our\/ 
$\,H\cong\bbZ_n\w$, acting on\/ $\,\tvs\hs$ linearly by\/ 
$\,H\times\tvs\ni(e^{i\theta}\nh,f)\mapsto f\circ r\nnh_\theta\w\in\tvs\nnh$, 
and\/ $\,L=\bbZ\times\bbZ_0^{\hskip-1.1ptH}\nnh$.
\item \,As a transformation of\/ $\,\tvs\nnh$, each\/ $\,(t,\psi)\in L\,$ equals 
the translation by\/ $\,t+\psi$.
\item $L\,$ consists of all translations by vectors\/ 
$\,\psi\hh'\in\bbZ\nh^H\,$ such that\/ 
$\,\psi\hh'_{\hskip-2.3pt\mathrm{avg}}\nh\in\bbZ$.
\end{enumerate}
An example of two mutually complementary\/ $\,H\nh$-in\-var\-i\-ant\/ 
$\,L$-sub\-spaces of\/ $\,\tvs\nh$, in the sense of\/ {\rm(\ref{cpl})} and 
Definition\/~{\rm\ref{lsbsp}}, is provided by the line\/ $\,\tvs\hh'$ of 
constant functions $\,H\to\bbR\,$ and the hyper\-plane\/ $\,\tvs\hh''$ 
consisting of all\/ $\,f:H\to\bbR\,$ with\/ 
$\,f\hskip-2.3pt_{\mathrm{avg}}\w=0$. The generic iso\-tropy groups 
$\,\stb\hh'\nnh,\stb\hh''\nnh\subseteq\bg\hs$ associated via\/ 
{\rm(\ref{krh})} with\/ $\,\tvs\hh'$ and\/ $\,\tvs\hh''$ are the translation 
groups\/ $\,\bbZ\times\nnh\{0\}\,$ and\/ 
$\,\{0\}\nnh\times\bbZ_0^{\hskip-1.1ptH}\nnh$, both contained in\/ $\,L$. 
Every coset of the\/ $\,L$-sub\-space\/ $\,\tvs\hh''$ is generic, cf.\ 
Section\/~{\rm\ref{gg}}.
}
\end{lemma}
\begin{proof}First, $\,\bg\hs$ acts on $\,\tvs\hs$ freely: if 
$\,f\circ r\nnh_{2\pi t}\w\nh+t+\psi=f\nh$, cf.\ (\ref{tpf}), with 
$\,f:H\to\bbR$, applying $\,(\hskip2.7pt)\nh_{\mathrm{avg}}\w$ to both sides, 
we get $\,t=0$, by (\ref{avg}), and hence 
$\,f\nnh\circ\nnh r\nnh_{2\pi t}\w\nh=f\nh$, so that the equality 
$\,f\nnh\circ\nnh r\nnh_{2\pi t}\w\nh+t+\psi=f\,$ reads $\,\psi=0$. Secondly, 
$\,H\,$ and $\,L\,$ described by (iii) arise from $\,\bg\hs$ as required in 
(\ref{lah}): the claim about $\,H\,$ is obvious, and so are (iv) -- (v), 
yielding the inclusion $\,L\subseteq\bg\hn\cap\tvs\nnh$. Conversely, 
$\,\bg\hn\cap\tvs\subseteq\hn L$. To verify this, suppose that 
$\,f\circ r\nnh_{2\pi t}\w\nh+t+\psi=f\nh+\psi\hh'$ for all 
$\,f\in\tvs\nh=\bbR\nnh^H\nnh$, some 
$\,(t,\psi)\in\bg\nnh$, and some $\,\psi\hh'\nh\in\tvs\nnh$. Taking the 
linear parts of both sides, we see that $\,t\in\bbZ\,$ and $\,(t,\psi)\in L$, 
as required.

Our $\,\bg\hs$ has a compact fundamental domain in $\,\tvs\nnh$, since so does 
the lattice $\,L\subseteq\bg\nnh$. Also, $\,\bg\hs$ must be 
tor\-sion-free: as $\,\bg\ni(t,\psi)\mapsto t\in\bbR\,$ is a group 
homo\-mor\-phism, any fi\-nite\hh-or\-der element $\,(t,\psi)\,$ of $\,\bg\hs$ 
has $\,t=0$, and so, by (\ref{tpf}), it acts via translation by $\,\psi$, 
which gives $\,\psi=0$. Next, to establish discreteness of the subset 
$\,\bg\hs$ of $\,\mathrm{Iso}\,\hs\tvs\hs$(and, consequently, (ii)), suppose 
that a sequence $\,(t_k\w,\psi\nh_k\w)\in\bg\hs$ with pairwise distinct terms 
yields, via (\ref{tpf}), a sequence convergent in 
$\,\mathrm{Iso}\,\hs\tvs\nnh$. Evaluating (\ref{tpf}) on $\,f\nh=0$, we get 
$\,(t_k\w,\psi\nh_k\w)\to(t,\psi)\,$ in $\,\bbR\times\bbR\nnh^H$ as 
$\,k\to\infty$, for some $\,(t,\psi)\,$ and, since 
$\,(t_k\w,\psi\nh_k\w)
\in[(\nh1\nh/\nh n\nh)\bbZ]\times\bbZ_0^{\hskip-1.1ptH}\nnh$, the sequence 
$\,(t_k\w,\psi\nh_k\w)\,$ becomes eventually constant, contrary to its terms' 
being pairwise distinct.

The final clause of the lemma follows since, by (iv), a $\,\bbZ\nh$-ba\-sis of 
$\,L\cap\tvs\hh'$ (or, $\,L\cap\tvs\hh''$) may be defined to consist just 
of the constant function $\,1\,$ (or, respectively, of the $\,\nmo$ 
functions $\,\psi\nnh_q:H\to\bbZ$, labeled by $\,q\in H\smallsetminus\{1\}$, 
where $\,\psi\nnh_q(q)=1=-\hh\psi\nnh_q(1)\,$ and $\hs\psi\nnh_q=0\hs$ on 
$\,H\smallsetminus\{1,q\}$. Specifically, 
$\,\psi\hh'\nh=\sum_q\w\psi\hh'(q)\hs\psi\nnh_q$ whenever 
$\,\psi\hh'\in\bbZ\nh^H$ and $\,\psi\hh'_{\hskip-2.3pt\mathrm{avg}}\nh=0$. 
Furthermore, the claims about $\,\stb\hh''$ and genericity of cosets follow 
from (\ref{tpf}) -- (\ref{avg}). The description of $\,\stb\hh'$ is in turn 
immediate if one identifies $\,\tvs\nnh/\hn\tvs\hh'$ with $\,\tvs\hh''$ and 
observes that the quotient action of $\,\bg\hs$ then becomes 
$\,\bg\times\tvs\hh''\ni((t,\psi),f)
\mapsto f\circ r\nnh_{2\pi t}\w\nh+\psi\in\tvs\hh''\nnh$. 

\end{proof}
The compact flat Riemannian manifold $\,\tvs\nnh/\nnh\bg\hs$ arising from our 
Bie\-ber\-bach group $\,\bg\hs$ as in Section~\ref{bg} is called the 
$\,n$-di\-men\-sion\-al {\smallit generalized Klein bottle\/} 
\cite[p.\ 163]{charlap}. The linear functional 
$\,\tvs\ni f\mapsto f\hskip-2.3pt_{\mathrm{avg}}\w\in\bbR\,$ is 
equi\-var\-i\-ant, due to (\ref{avg}), with respect to the actions of 
$\,\bg\hs$ and $\,\bbZ$ (the latter, on $\,\bbR$, via translations by 
multiples of $\,\nh1\nh/\nh n$), relative to the homo\-mor\-phism 
$\,\bg\ni(t,\psi)\mapsto t\in(\nh1\nh/\nh n\nh)\bbZ$. Thus, it descends, in 
view of Remark~\ref{covpr}(c), to a bundle projection 
$\,\tvs\nnh/\nnh\bg\to\bbR/[(\nh1\nh/\nh n\nh)\bbZ]$, making 
$\,\tvs\nnh/\nnh\bg\hs$ a bundle of tori over the circle. The fibres of this 
bundle are, obviously, the images, under the projection 
$\,\tvs\nh\to\tvs\nnh/\nnh\bg\nh$, of cosets of the $\,L$-sub\-space 
$\,\tvs\hh''\nh\subseteq\tvs\hs$ mentioned in the final clause of 
Lemma~\ref{biebg}, all of them generic. On the other hand, $\,\tvs\hh'$ has 
some non\-ge\-ner\-ic cosets -- an example is $\,\tvs\hh'$ itself, with the 
isotropy group easily seen to be $\,[(\nh1\nh/\nh n\nh)\bbZ]\times\nnh\{0\}$.

The $\,n$-di\-men\-sion\-al generalized Klein bottle, for any $\,n\ge2$, shows 
that the last inclusion of Theorem~\ref{restr}(ii-c) may be proper. 
In fact, the isotropy group $\,[(\nh1\nh/\nh n\nh)\bbZ]\times\nnh\{0\}\,$ of 
the preceding paragraph, although not contained in the lattice $\,L$, acts on 
$\,\tvs\hh'$ by translations.

\section{Remarks on holonomy}\label{rh}
The correspondence -- Remark~\ref{bijct} -- between Bie\-ber\-bach groups and 
compact flat manifolds has an extension to al\-most-Bie\-ber\-bach groups and 
in\-fra-nil\-man\-i\-folds \cite{dekimpe} obtained by using (instead of the 
translation vector space of a Euclidean af\-fine space) a connected, simply 
connected nil\-po\-tent Lie group $\,\sg\,$ acting simply transitively 
on a manifold $\,\mathcal{E}\nh$, and replacing the Bie\-ber\-bach group with 
a tor\-sion-free uniform discrete sub\-group 
$\,\bg\hs$ of $\,\mathrm{Dif{}f}\,\mathcal{E}$ contained in a 
sem\-i\-di\-rect product, canonically transplanted so as to act on 
$\,\mathcal{E}\nh$, of $\,\sg\,$ and a maximal compact sub\-group of 
$\,\mathrm{Aut}\,\sg\nh$. Here `uniform' means admitting a compact fundamental 
domain, cf.\ Remark~\ref{cptfd}. The analogs of (\ref{exa}) and (\ref{cmp}) 
remain valid, reflecting the fact that any in\-fra-nil\-man\-i\-fold is the 
quotient of a nil\-man\-i\-fold under the action of a finite group $\,H\nnh$.

A somewhat similar picture may arise in some cases where $\,\sg\,$ is not 
assumed nil\-po\-tent. As an example, one has 
$\,\sg\cong\mathrm{Spin}\hs(m,1)$, the universal covering group of the 
identity component $\,\sg/\bbZ_2\w\cong\mathrm{SO}\hn^+\nnh(m,1)\,$ of the 
pseu\-\hbox{do\hs-}\hskip0ptor\-thog\-onal group in an 
$\,(m+1)$-di\-men\-sion\-al Lo\-rentz\-i\-an vector space $\,\mathcal {L}$, 
$\,m\ge3$. Here $\,\mathcal{E}\hs$ is the (two\hs-fold) universal covering 
manifold of the or\-tho\-nor\-mal-frame bundle of the future unit 
pseu\-do\-sphere $\,\mathcal{S}\subseteq\hn\mathcal{L}$, isometric to the 
hyperbolic $\,m$-space, and $\,\sg/\bbZ_2\w$ acts on $\,\mathcal{S}\,$ via 
hyperbolic isometries, leading to an action of $\,\sg\,$ on 
$\,\mathcal{E}\nh$. The or\-tho\-nor\-mal-frame bundles of compact hyperbolic 
manifolds obtained as quotients of $\,\mathcal{S}\,$ give rise to the 
required tor\-sion-free uniform discrete sub\-groups $\,\bg\nh$.

The resulting compact quotient manifolds $\,\mathcal{M}=\mathcal{E}/\hn\bg\hs$ 
can be endowed with various interesting Riemannian metrics coming from 
$\,\bg\nh$-in\-var\-i\-ant metrics on $\,\mathcal{E}\nh$. For $\,\bg\hs$ and 
$\,\mathcal{E}\hs$ of the preceding paragraph, a particularly natural choice 
of an invariant {\smallit indefinite\/} metric is provided by the Kil\-ling 
form of $\,\sg\nh$, turning $\,\mathcal{M}\,$ into a compact locally symmetric 
pseu\-\hbox{do\hskip1pt-}\hskip-.7ptRiem\-ann\-i\-an Ein\-stein manifold.

Outside of the Bie\-ber\-bach-group case, however, these metrics are not flat, 
and finite groups $\,H\,$ such as mentioned above cannot serve as their 
holonomy groups. The holonomy interpretation of $\,H\,$ still makes sense, 
though, if -- instead of metrics -- one uses $\,\bg\nh$-in\-var\-i\-ant flat 
connections, with (parallel) torsion, on $\,\mathcal{E}\nh$. Two such 
standard connections are naturally induced by bi-in\-var\-i\-ant connections 
on $\,\sg\nh$, characterized by the property of making all 
left-in\-var\-i\-ant (or, right-in\-var\-i\-ant) vector fields parallel. Both 
of these connections are, due to their naturality, invariant under all 
Lie-group auto\-mor\-phisms of $\,\sg\nh$. 

\setcounter{section}{1}
\renewcommand{\thesection}{\Alph{section}}
\setcounter{theorem}{0}
\renewcommand{\thetheorem}{\thesection.\arabic{theorem}}
\section*{Appendix: Hiss and Szczepa\'nski's reducibility theorem}\label{hs}
Let us consider an {\smallit abstract Bie\-ber\-bach group}, that is, any 
tor\-sion-free group $\,\bg\hs$ containing a finitely generated free 
A\-bel\-i\-an normal sub\-group $\,L\,$ of a finite index, which is at the 
same time a maximal A\-bel\-i\-an sub\-group of $\,\bg\nh$. As shown by 
Zas\-sen\-haus \cite{zassenhaus}, up to isomorphisms these groups coincide 
with the Bie\-ber\-bach groups of Section~\ref{bg}, and one can again 
summarize their structure using the short exact sequence
\begin{equation}\label{abs}
L\,\to\,\hs\bg\,\to\,H,\hskip12pt\mathrm{where\ }\,\,H\,=\,\bg/\nh L\hh.
\end{equation}
For the tensor product $\,G\otimes G'$ of A\-bel\-i\-an groups $\,G,G'$ one 
has canonical isomorphisms
\begin{equation}\label{can}
\bbZ\otimes G\cong G\hh,\hskip7pt(G_1\w\nnh\oplus G_2\w)\otimes G'
\cong(G_1\w\nnh\otimes G')\oplus(G_2\w\nnh\otimes G')\hh,\hskip7pt
L\otimes\nh\bbQ\cong\mathrm{Hom}\hs(L\nh^*\nh,\bbQ)\hh,
\end{equation}
where $\,L\nh^*\nh=\mathrm{Hom}\hs(L,\bbZ)\,$ and, for simplicity, $\,L\,$ is 
assumed to be finitely generated and free. In the last case, with a suitable 
integer $\,n\ge0$, there are noncanonical isomorphisms
\begin{equation}\label{ncn}
\mathrm{a)}\hskip6ptL\cong\bbZ^n,\hskip22pt\mathrm{b)}\hskip6ptL\otimes\nh\bbQ
\cong\bbQ\hn^n\nh,
\end{equation}
while, using the injective homo\-mor\-phism 
$\,L\ni\lambda\mapsto\lambda\otimes\nh1\in L\otimes\nh\bbQ\,$ to treat 
$\,L\,$ as a sub\-group of $\,L\otimes\nh\bbQ$, we see that, under suitably 
chosen identifications (\ref{ncn}),
\begin{equation}\label{crs}
\mathrm{the\ inclusion\ }\,L\subseteq L\otimes\nh\bbQ\,\mathrm{\ corresponds\ 
to\ the\ standard\ inclusion\ }\,\bbZ^n\nh\subseteq\hs\bbQ\hn^n\nh.
\end{equation}
Finally, if $\,L\,$ as above is a (full) lattice in an 
fi\-nite-di\-men\-sion\-al real vector space $\,\tvs$ (cf.\ 
Remark~\ref{lttce}), a further canonical isomorphic identification arises:
\begin{equation}\label{spn}
L\otimes\nh\bbQ\cong\mathrm{Span}_\bbQ\w L\hh,
\end{equation}
that is, we may view $\,L\otimes\nh\bbQ\,$ as the rational vector subspace of 
$\,\tvs$ spanned by $\,L$.

Let $\,\bg\hs$ now be an abstract Bie\-ber\-bach group. Hiss and Szczepa\'nski 
\cite[the corollary in Sect.\ 1]{hiss-szczepanski} proved that, if $\,L\,$ 
in (\ref{abs}) satisfies (\ref{ncn}.a) with $\,n\ge2$, then {\smallit the} 
(obviously $\,\bbQ$-lin\-e\-ar) {\smallit action of\/} $\,H\,$ {\smallit on\/} 
$\,L\otimes\nh\bbQ\,$ {\smallit must be reducible}, in the sense of admitting 
a nonzero proper invariant rational vector sub\-space $\,\mathcal{W}\nh$.

Next, using (\ref{crs}), we may write 
$\,L=\bbZ^n\nh\subseteq\hs\bbQ\hn^n\nh=L\otimes\nh\bbQ$, so that 
$\,\mathcal{W}\subseteq\hs\bbQ\hn^n\subseteq\rn\nh$, and the closure 
$\,\tvs\hh'$ of $\,\mathcal{W}\hs$ in $\,\rn$ has the real dimension 
$\,\dim_\bbQ\w\hskip-3pt\mathcal{W}\hs$ (any $\,\bbQ$-ba\-sis of 
$\,\mathcal{W}\hs$ being, obviously, an $\,\bbR$-ba\-sis of $\,\tvs\hh'$). By 
clearing denominators, one can replace such a $\,\bbQ$-ba\-sis with one 
consisting of vectors in $\,L=\bbZ^n\nh$, and so, by Remark~\ref{lttce}(b), 
the intersection $\,L\nh'\nh=L\cap\mathcal{W}\nh=L\cap\tvs\hh'$ is a lattice 
in $\,\tvs\hh'\nnh$. In other words, we obtain (\ref{qdr}).

A stronger version of Hiss and Szczepa\'nski's reducibility theorem was more 
recently established by Lutowski \cite{lutowski}.

\end{document}